\documentclass[11pt]{amsart}
\usepackage{amsmath,amssymb,amsfonts,amsthm}
\usepackage{ifpdf}

\usepackage{comment}

\usepackage[all]{xy}
\SelectTips{cm}{}
\usepackage[section]{placeins}
\usepackage{leftidx}
\usepackage{stmaryrd} % additional math symbols, e.g. \mapsfrom
\usepackage{microtype}

\newcommand{\pathtotrunk}{./}
\newcommand{\pathtodiagrams}{\pathtotrunk}

\newcommand{\mathfig}[2]{\ensuremath{\hspace{-3pt}\begin{array}{c}%
  \raisebox{-2.5pt}{\includegraphics[width=#1\textwidth]{\pathtodiagrams #2}}%
\end{array}\hspace{-3pt}}}

\newcommand{\arxiv}[1]{\href{http://arxiv.org/abs/#1}{\tt arXiv:\nolinkurl{#1}}}
\newcommand{\doi}[1]{\href{http://dx.doi.org/#1}{{\tt DOI:#1}}}

\newcommand{\googlebooks}[1]{(preview at \href{http://books.google.com/books?id=#1}{google books})}

\theoremstyle{plain}
\newtheorem{prop}{Proposition}[subsection]
\makeatletter
\@addtoreset{prop}{section}
\makeatother

\newtheorem{thm}[prop]{Theorem}
\newtheorem{lem}[prop]{Lemma}
\newtheorem*{lem*}{Lemma}
\newtheorem{cor}[prop]{Corollary}
\newtheorem*{cor*}{Corollary}

\newtheorem{defn}[prop]{Definition}         % numbered definition
\newtheorem*{defn*}{Definition}             % unnumbered definition

\newenvironment{rem}{\vspace{0.3cm}\noindent\textsl{Remark.}}{}  % perhaps looks better than rem above?
\newenvironment{example}{\vspace{0.3cm}\noindent\textbf{Example.}}{}  % perhaps looks better than rem above?
\newtheorem{rem*}[prop]{Remark}
\numberwithin{equation}{section}
\newcommand{\noop}[1]{}

\def\clap#1{\hbox to 0pt{\hss#1\hss}}

\renewcommand{\imath}{\mathfrak{i}}
\renewcommand{\jmath}{\mathfrak{j}}

\newcommand{\ssum}[1]{\Sigma#1}
\newcommand{\sumhat}{\overline{\sum}}
\newcommand{\sumtah}{\underline{\sum}}

\newcommand{\iso}{\cong}

\newcommand{\qi}[2][q]{\left[#2\right]_{#1}}
\newcommand{\qBinomial}[3][q]{\genfrac{[}{]}{0pt}{}{#2}{#3}_{#1}}

\newcommand{\directSum}{\oplus}
\newcommand{\DirectSum}{\bigoplus}
\newcommand{\tensor}{\otimes}

\newcommand{\Hom}{\operatorname{Hom}}
\newcommand{\End}[1]{\operatorname{End}\left(#1\right)}

\usepackage{hyperref}
\usepackage{graphicx}

\usepackage{tikz}
\usetikzlibrary{shapes}
\usetikzlibrary{backgrounds}
\usetikzlibrary{decorations,decorations.pathreplacing,decorations.markings}
\usetikzlibrary{fit,calc,through}
\usetikzlibrary{external}

\tikzstyle{mid>}=[decoration={markings, mark=at position 0.5 with {\arrow{>}}}, postaction={decorate}]
\tikzstyle{mid<}=[decoration={markings, mark=at position 0.5 with {\arrow{<}}}, postaction={decorate}]
\tikzstyle{upper>}=[decoration={markings, mark=at position 0.8 with {\arrow{>}}}, postaction={decorate}]
\tikzstyle{upper<}=[decoration={markings, mark=at position 0.8 with {\arrow{<}}}, postaction={decorate}]
\tikzstyle{lower>}=[decoration={markings, mark=at position 0.2 with {\arrow{>}}}, postaction={decorate}]
\tikzstyle{lower<}=[decoration={markings, mark=at position 0.2 with {\arrow{<}}}, postaction={decorate}]

\def\Foam{{\mathcal{F}{\rm oam}}}
\newcommand{\alt}{\wedge}
\newcommand{\Alt}[2]{{\textstyle\bigwedge^{#1}_{#2}}}

\newcommand{\one}{1}

\def\l{\lambda}
\def\bZ{{\mathbb{Z}}}
\def\sl{{\mathfrak{sl}}}
\def\Sp{{\mathcal{S}p}}
\def\FSp{{\mathcal{FS}p}}
\def\bC{{\mathbb{C}}}

\def\SL{{\rm{SL}}}

\def\gl{{\mathfrak{gl}}}
\def\dU{\dot{{\mathcal{U}}}_q}

\def\Rep{\mathcal{R}ep}
\def\la{\langle}
\def\ra{\rangle}
\def\dalg{\dot{{U}}_q}

\newcommand{\ul}[1]{{\underline{#1}}}

\newcommand{\Lad}{\mathcal{L}ad}

\usepackage{environ}
\usepackage{xargs}

\newcommandx{\NewEnvironx}[5][2,3]{%
  \expandafter\newcommandx\csname start#1\endcsname[#2][#3]{#4}%
  \NewEnviron{#1}{\csname start#1\expandafter\endcsname\BODY #5}}

\newcommand{\ladderX}{1.5}
\newcommand{\ladderY}{1.5}
\newcommand{\ladderR}{0.6}
\newcommand{\laddercoordinates}[2]{
\foreach \x in {0,...,#1} {
	\foreach \y in {0,...,#2} {
		\coordinate (l\x\y) at (\x * \ladderX, \y * \ladderY);
		\coordinate (u\x\y) at ($(l\x\y)+\ladderR*(0,\ladderY)$);
		\coordinate (d\x\y) at ($(l\x\y)+(0,\ladderY)-\ladderR*(0,\ladderY)$);
	}
}
}
\newcommand{\ladderEn}[5]{
\draw[mid>] (l#1#2) -- (d#1#2);
\draw[mid>] (d#1#2) -- ($(l#1#2)+(0,\ladderY)$) node[left] {#3};
\draw[mid>] ($(l#1#2)+(\ladderX,0)$) -- ($(u#1#2)+(\ladderX,0)$);
\draw[mid>] ($(u#1#2)+(\ladderX,0)$) -- ($(l#1#2)+(\ladderX,\ladderY)$) node[right] {#4};
\draw[mid>] (d#1#2) --node[above]{#5} ($(u#1#2)+(\ladderX,0)$);
}
\newcommand{\ladderE}[4]{\ladderEn{#1}{#2}{#3}{#4}{}}
\newcommand{\ladderFn}[5]{
\draw[mid>] (l#1#2) -- (u#1#2);
\draw[mid>] (u#1#2) -- ($(l#1#2)+(0,\ladderY)$) node[left] {#3};
\draw[mid>] ($(l#1#2)+(\ladderX,0)$) -- ($(d#1#2)+(\ladderX,0)$);
\draw[mid>] ($(d#1#2)+(\ladderX,0)$) -- ($(l#1#2)+(\ladderX,\ladderY)$) node[right] {#4};
\draw[mid>] ($(d#1#2)+(\ladderX,0)$) --node[above]{#5} (u#1#2);
}

\newcommand{\ladderIn}[3]{\draw[mid>] (l#1#2) -- +($#3*(0,\ladderY)$);}
\newcommand{\ladderI}[2]{\ladderIn{#1}{#2}{1}}

\NewEnvironx{ladder}[2]{%
  \begin{tikzpicture}[baseline=13*\ladderY*#2]\laddercoordinates{#1}{#2}}
{\end{tikzpicture}}

\newcommand{\fuse}[3]{\tikz[baseline=0.5cm]{
\coordinate (z1) at (0,0);
\coordinate (z2) at (1,0);
\coordinate (c) at (0.5,0.5);
\coordinate (e) at (0.5,1);
\draw[mid>] (z1) node[below] {$#1$} -- (c);
\draw[mid>] (z2) node[below] {$#2$} -- (c);
\draw[mid>] (c) -- (e) node[above] {$#3$};
}}
\newcommand{\fork}[3]{\tikz[baseline=0.5cm]{
\coordinate (z1) at (0,1);
\coordinate (z2) at (1,1);
\coordinate (c) at (0.5,0.5);
\coordinate (e) at (0.5,0);
\draw[mid<] (z1) node[above] {$#1$} -- (c);
\draw[mid<] (z2) node[above] {$#2$} -- (c);
\draw[mid<] (c) -- (e) node[below] {$#3$};
}}

\def\semicolon{;}
\def\applytolist#1{
    \expandafter\def\csname multi#1\endcsname##1{
        \def\multiack{##1}\ifx\multiack\semicolon
            \def\next{\relax}
        \else
            \csname #1\endcsname{##1}
            \def\next{\csname multi#1\endcsname}
        \fi
        \next}
    \csname multi#1\endcsname}

\def\calc#1{\expandafter\def\csname c#1\endcsname{{\mathcal #1}}}
\applytolist{calc}QWERTYUIOPLKJHGFDSAZXCVBNM;

\usepackage{color}

\usepackage{xcolor}
\definecolor{dark-red}{rgb}{0.7,0.25,0.25}
\definecolor{dark-blue}{rgb}{0.15,0.15,0.55}
\definecolor{medium-blue}{rgb}{0,0,0.65}
\hypersetup{
    colorlinks, linkcolor={dark-red},
    citecolor={dark-blue}, urlcolor={medium-blue}
}

\title{Webs and quantum skew Howe duality}

\author{Sabin~Cautis}
\email{cautis@usc.edu}
\address{Department of Mathematics\\ University of British Columbia \\ Vancouver, Canada}

\author{Joel~Kamnitzer}
\email{jkamnitz@math.toronto.edu}
\address{Department of Mathematics\\ University of Toronto \\ Toronto, Canada}

\author{Scott~Morrison}
\email{scott.morrison@anu.edu.au}
\address{Mathematical Sciences Institute \\ The Australian National University \\ Canberra, Australia}

\begin{document}

\makeatletter
\@addtoreset{equation}{section}
\gdef\theequation{\thesection.\arabic{equation}}
\makeatother

\begin{abstract}
We give a diagrammatic presentation in terms of generators and relations of the representation category of $U_q(\sl_n) $. More precisely, we produce all the relations among $\SL_n$-webs, thus describing the full subcategory $\otimes$-generated by fundamental representations $\Alt{k}{} \mathbb C^n$ (this subcategory can be idempotent completed to recover the entire representation category). Our result answers a question posed by Kuperberg \cite{MR1403861} and affirms conjectures of Kim \cite{math.QA/0310143} and Morrison \cite{0704.1503}. Our main tool is an application of quantum skew Howe duality. %This is the published version of \arxiv{1210.6437}.
\end{abstract}

\maketitle

\hypersetup{
    linkcolor={black},
    citecolor={dark-blue}, urlcolor={medium-blue}
}

\tableofcontents

\hypersetup{
    linkcolor={dark-red},
    citecolor={dark-blue}, urlcolor={medium-blue}
}

\section{Introduction}

The representation theory of $\SL_n$ is a pivotal tensor category, and it is natural to ask for a presentation by generators and relations, as a pivotal tensor category.

There are two choices one needs to make. First, one can pass to a full subcategory whose idempotent completion recovers the entire representation category. In particular, in this paper we look at the full subcategory, denoted $\Rep(\SL_n)$, whose objects are isomorphic to tensor products of the fundamental representations $\Alt{k}{} \mathbb C^n$ of $\sl_n$.

Second, we need to decide which generators to use. We take the natural maps
$$\Alt{k}{} \bC^n \tensor \Alt{l}{} \bC^n \rightarrow \Alt{k+l}{} \bC^n \ \ \text{ and } \ \ \Alt{k+l}{} \bC^n \rightarrow \Alt{k}{} \bC^n \tensor \Alt{l}{} \bC^n$$
which we depict diagrammatically as follows (where we read from the bottom up)
\begin{equation} \label{eq:1}
\fuse{k}{l}{k+l} \ \ \text{ and } \ \ \fork{k}{l}{k+l}.
\end{equation}
It is relatively easy to show that these are indeed generators, {\it i.e.} that every $\SL_n$-linear map between tensor products of fundamental representations can be written as tensor products and compositions of these maps, along with the duality, pairing, and copairing maps \cite[Prop. 3.5.8]{0704.1503}. The question then, is to identify all the relations between compositions of these generators.

Said another way, we have a pivotal category, the ``free spider category'' $\FSp(\SL_n) $, of trivalent webs made up by glueing together the pieces in (\ref{eq:1}). The edges in these webs are oriented and labelled by $\{1, \ldots, n-1\}$. Moreover, we have a full and dominant functor $ \FSp(\SL_n) \rightarrow \Rep(\SL_n) $. The question is to identify the pivotal ideal which is the kernel of this functor.

We completely answer this question (Theorem \ref{thm:main}) by giving generators of this pivotal ideal in section \ref{sec:spider}, equations (\ref{eq:switch}) -- (\ref{eq:commutation}).

\subsection{Some history}
This problem has been studied previously. For $n=2$, there are no trivalent vertices and we do not need to label strands since the only label they could carry is $1$. So, in this case, the free spider category is essentially just the category of embedded 1-manifolds up to isotopy. The kernel of the functor to representation theory is the ideal generated by the relation $\tikz[baseline=-2pt]{\node[draw,circle] {};} = 2$. If we were to ignore orientations, this gives us the Temperley-Lieb category where the objects are indexed by ${\mathbb N}$ and the morphisms are crossingless matchings.

For $n=3$, relations generating the kernel were determined by Kuperberg \cite{MR1403861}:
\begin{align*}
\mathfig{0.04}{clockwise_circle} & = 3   &
\mathfig{0.045}{bigon} & = -2 \mathfig{0.01}{strand}
\end{align*}
\begin{align*}
\mathfig{0.085}{oriented_square} & = \mathfig{0.085}{two_strands_horizontal} + \mathfig{0.085}{two_strands_vertical}.
\end{align*}
They allow one to remove circles, bigons, and squares. He introduced the term ``$\SL_3$ spider" for the resulting diagrammatic category.

For $n \geq 4$, relations generating the kernel were proposed by Kim in \cite{math.QA/0310143} (for $n=4$) and by the third author in \cite{0704.1503} (for any $n$). However, they did not prove that their lists of relations are complete.

\subsection{Main result}
We define the $\SL_n$-spider, denoted $\Sp(\SL_n) $, to be the quotient of $ \FSp(\SL_n) $ by the relations (\ref{eq:switch}) -- (\ref{eq:commutation}), which all involve webs with $\le $ 4 boundary edges.  Our main result (Theorem \ref{thm:main}) states that $ \Sp(\SL_n) $ is equivalent to $ \Rep(\SL_n) $.  In particular this shows that the relations in \cite{0704.1503} are complete.  Quite surprisingly, we do not need the most complicated relations from \cite{0704.1503}.

\subsection{Skew Howe duality and webs}
The core idea of our proof is to use skew Howe duality.  We give a recipe for the relations in $\Rep(\SL_n)$ as certain truncations of relations holding in $U(\gl_m)$, for $m$ sufficiently large. We now give a quick overview of the argument.

Consider the commuting actions of $U(\sl_n)$ and $U(\gl_m)$ on $\Alt{\bullet}{}(\mathbb{C}^n \tensor \mathbb{C}^m)$.  Skew Howe duality tells us that the resulting map
\begin{equation}\label{eq:2}
\dot{U}(\gl_m) \rightarrow \bigoplus_K \operatorname{End}_{U(\sl_n)}\left(\Alt{K}{}(\bC^n \tensor \bC^m)\right)
\end{equation}
is surjective, where $\dot{U}(\gl_m)$ denotes Lusztig's idempotent form. Moreover, we prove that its kernel is the ideal generated by those weight space idempotents falling outside the weight support of $\Alt{\bullet}{}(\mathbb C^n \tensor \mathbb C^m)$.  This result is proved in section \ref{sec:fully-faithful}. The quotient of $\dot{U}(\gl_m)$ by this ideal is denoted $\dot{U}^n(\gl_m)$.

Now, as $U(\sl_n)$-representations, we have
\begin{equation}\label{eq:3}
\Alt{K}{}(\mathbb C^n \tensor \mathbb C^m)  = \Alt{K}{}(\bC^n \directSum \cdots \directSum \bC^n) = \DirectSum_{\underline{k}: \sum \underline{k} = K} \Alt{k_1}{} \bC^n \tensor \cdots \tensor \Alt{k_m}{} \bC^n.
\end{equation}

Thus combining (\ref{eq:2}) and (\ref{eq:3}) we obtain an isomorphism
\begin{equation}\label{eq:4}
\dot{U}^n(\gl_m) \xrightarrow{\sim} \bigoplus_{\ul{k}, \ul{l} : \sum \ul{k} = \sum \ul{l}} \Hom_{U(\sl_n)}\left(\Alt{k_1}{} \mathbb C^n \tensor \cdots \tensor \Alt{k_m}{} \mathbb C^n, \Alt{l_1}{} \mathbb C^n \tensor \cdots \tensor \Alt{l_m}{} \mathbb C^n\right)
\end{equation}
Under this map, elements of $U(\gl_m)$ are sent to particular webs, which we call ladders. Diagram (\ref{eq:5}) illustrates the image of $F_1^{(r)}F_2^{(t)}E_1^{(s)} \one_{k_1k_2k_3} \in \dot{U}(\gl_3)$ under this map.
\begin{figure}[ht]
\begin{equation}\label{eq:5}
\begin{ladder}{2}{3}
\node[left] at (l00) {$k_1$};
\node[left] at (l10) {$k_2$};
\node[left] at (l20) {$k_3$};
\ladderFn{0}{0}{$k_1{+}s$}{$k_2{-}s$}{$s$}
\ladderEn{1}{1}{\small $k_2{-}s{-}t$}{$k_3{+}t$}{$t$}
\ladderEn{0}{2}{$k_1{+}s{-}r$}{\small $k_2{-}s{-}t{+}r$}{$r$}
\ladderI{0}{1}
\ladderI{2}{0}
\ladderI{2}{2}
\end{ladder}
\end{equation}
 \caption{Ladder corresponding to $F_1^{(r)}F_2^{(t)}E_1^{(s)} \one_{k_1k_2k_3} \in \dot{U}(\gl_3)$ (reading bottom to top in the diagram, right to left in the word).}
 \end{figure}

This allows us to write the generating relations of $\dot{U}^n(\gl_m) $ in a diagrammatic form. By the above isomorphism, we see that these diagrammatic relations become the generating relations in $\Rep(\SL_n)$.

\subsection{The quantum deformation and categorification}

Our whole discussion above has a natural $q$-deformation. In other words, $\Rep(\SL_n)$ becomes the category of $U_q(\sl_n)$-modules generated by tensor products of fundamental representations, while ${\dot U}(\gl_m)$ is replaced by the quantum group $\dalg(\gl_m)$.

In the previous section we assumed $q=1$ in order to simplify the notation. However, in the rest of the paper we will always consider this quantum deformation, with $ q $ taken as a formal variable. In section \ref{sec:quantumskew} we will discuss in detail the quantum skew Howe duality results we need, which may be of independent interest.

The $q$-deformation of the linear maps (\ref{eq:2}) can be rephrased as giving us a functor $ \Psi_m^n : \dU(\gl_m) \rightarrow \Sp(\SL_n) $ where $ \dU(\gl_m) $ is the categorical version of $ \dalg(\gl_m) $ (replacing weight space idempotents with distinct objects).  The categories $\dU(\gl_m)$ and $\Sp(\SL_n)$ can both be categorified. First, $\dU(\gl_m)$ can be lifted to the 2-category $\dot{\mathbb{U}}_q(\gl_m)$, defined by Khovanov-Lauda \cite{MR2628852} and Rouquier \cite{0812.5023}.  On the other hand, the spider category $\Sp(\SL_n)$ can be lifted to a 2-category $\Foam_n$ where the 2-morphisms consist of ``foams'' between webs, introduced by Khovanov \cite{MR2100691} (for n = 3) and partially described for general $n$ by Mackaay, Stosic, and Vaz \cite{MR2491657}.

Thus, it is natural to ask if the functor $ \Psi_m^n $ lifts to a 2-functor
$$\widetilde{\Psi}_m^n: \dot{\mathbb{U}}_q(\gl_m) \rightarrow \Foam_n.$$
In the cases $n=2$ and $3$ such a 2-functor has recently been studied by Mackaay, Pan and Tubbenhauer \cite{1206.2118} and by Lauda, Rose and Queffelec \cite{1212.6076}.

There is something more specific one can say. It is easy to see that the functor $\Psi_m^n$ factors through the quotient $\dU^n(\gl_m)$ of $\dU(\gl_m)$ by the ideal generated by the identity morphisms $1_\l$ where $\l$ is not $n$-bounded (see section \ref{sec:idemform} for the definition of $n$-bounded weights). The point of taking this quotient is that, combining Theorems \ref{th:functorfullyfaithful} and \ref{thm:main} together with diagram (\ref{diag:main}), we have that 
$$\Psi_m^n: \dU^n(\gl_m) \rightarrow \Sp(\SL_n)$$
is fully faithful. Thus the spider category is essentially a quotient of the limit category $\dU(\gl_\infty)$
$$\Sp(\SL_n) \iso \dU^n(\gl_\infty).$$

One can likewise define $\dot{\mathbb{U}}_q^n(\gl_m)$ as the quotient of $\dot{\mathbb{U}}_q(\gl_m)$ by the ideal generated by the identity 2-morphisms ${\rm id}_{1_\l}: 1_\l \rightarrow 1_\l$ where $\l$ is not $n$-bounded. It should be easy to check that $\widetilde{\Psi}_m^n$ factors through this quotient. We then speculate that the 2-functor
$$\widetilde{\Psi}_m^n: \dot{\mathbb{U}}_q^n(\gl_m) \rightarrow \Foam_n$$
is fully faithful on 2-morphisms. This would mean that the foam categories are essentially alternative descriptions of certain quotients of the limit 2-category $\dot{\mathbb{U}}_q(\gl_\infty)$.

\subsection{Braiding and knot invariants}
The category $\Rep(\SL_n)$ is in fact a braided monoidal category where the braiding comes from the $R$-matrix of the quantum group $ U_q(\sl_n)$.  In section \ref{se:braiding}, we express this braiding using webs.

More precisely, we define a braided monoidal category structure on $ \dU^n(\gl_\bullet) = \oplus_m \dU^n(\gl_m) $ using Lusztig's quantum Weyl groups elements.  We then prove (Theorem \ref{th:braiding}) that the functor $ \Phi^n : \dU^n(\gl_\bullet) \rightarrow \Rep(\SL_n) $ carries the braiding in $ \dU^n(\gl_\bullet) $ to the braiding in $ \Rep(\SL_n) $.  In particular, this shows that under quantum skew Howe duality, Lusztig's quantum group element $ T \in \dU(\gl_2) $ is taken to the $R$-matrix braidings $ \beta : \Alt{k}{q} \bC^n_q \otimes \Alt{l}{q} \bC^n_q \rightarrow \Alt{l}{q} \bC^n_q \otimes \Alt{k}{q} \bC^n_q$.  This last fact was previously proven by the first and second authors and Licata in \cite{MR2593278}, following the approach from \cite{MR1896470} who established the analogous result for symmetric Howe duality.  The proofs from \cite{MR1896470} and \cite{MR2593278} involve a somewhat lengthy computation, whereas the categorical approach in this paper provides a far more conceptual proof.

The braided monoidal category structure on $\Rep(\SL_n) $ leads to quantum knot invariants.  The results of this paper show that these knot invariants can be computed using webs. Of course, this was known by the work of Murakami, Ohtsuki and Yamada in \cite{MR1659228}. Our work extends theirs in the sense that while they gave a partition function evaluating closed webs, we have shown that the category of open webs, with the relations described here, is itself equivalent to the representation category. We would like to have an evaluation algorithm, showing directly that any closed web can be evaluated to a scalar by repeated application of the relations here; one has been proposed by Jeong and Kim in \cite{math/0506403}. Grant in \cite{1212.4511} also gives a somewhat indirect evaluation algorithm using a subset of our relations. While the tensor category defined by those relations is thus evaluable, it seems possible that it is degenerate (\emph{i.e.} there are negligible morphisms, which when paired with any other morphism to give a closed diagram, give zero) and hence not equivalent to the representation category. Sikora's work in \cite{MR2171796} gives an alternative presentation of the representation category as a \emph{braided} pivotal category, using only the standard representation and the the determinant map $\tensor^n \bC^n_q \to \Alt{n}{q} \bC^n_q$ and its dual as generators. Again, it is not clear whether the diagram category with the given relations is degenerate or not.

\vspace{.5cm}

\noindent {\bf Acknowledgments:}
The authors benefited from discussions with Arkady Berenstein, Bruce Fontaine, Stavros Garoufalidis, Greg Kuperberg, Valerio Toledano Laredo and Sebastian Zwicknagl. We would like to thank Dongho Moon who pointed out that the relation of Equation \eqref{eq:tag-migration2} was missing in the first 3 versions posted on the arXiv! S.C. was supported by NSF grant DMS-1101439 and the Alfred P. Sloan foundation, S.M. was supported by the Australian Research Council grant DE120100232 and by DOD-DARPA HR0011-12-1-0009, J.K. was supported by NSERC.  We would also like to thank VIA Rail for providing the venue where much of this research was carried out.

\section{The categories \texorpdfstring{$\FSp(\SL_n)$}{FSp(SL\_n)} and \texorpdfstring{$\Sp(\SL_n)$}{Sp(SL\_n)}}\label{sec:diagrams}

We will denote by $[n]_q$ the quantum integer $q^{n-1} + q^{n-3} + \dots + q^{-n+3} + q^{-n+1}$. More generally, we have quantum binomial coefficients
$$\qBinomial{n}{k} := \frac{[n]_q\dots[1]_q}{([n-k]_q \dots [1]_q)([k]_q \dots [1]_q)}.$$
We also adopt the convention that $[-n]_q = -[n]_q$, which is clear if we write $[n]_q = \frac{q^{n} - q^{-n}}{q-q^{-1}}$.

\subsection{The free spider category  \texorpdfstring{$\FSp(\SL_n)$}{FSp(SL\_n)} }
The free spider category $\FSp(\SL_n)$ has as objects sequences $\ul{k}$ in $\{1^\pm,\ldots,(n-1)^\pm\}$, and as morphisms ($\bC(q) $-linear combinations of) oriented planar graphs locally modeled on the following four types of vertices:
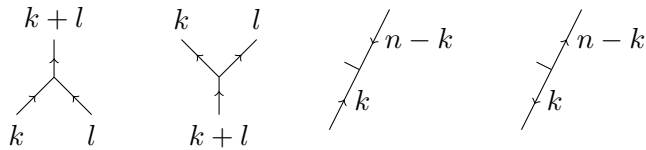
\begin{align*}
\fuse{k}{l}{k+l}
\qquad
\fork{k}{l}{k+l}
\qquad
\tikz[baseline=0.7cm]{
\foreach \n in {0,1,2} {
	\coordinate (a\n) at (0.4*\n, 0.8*\n);
}
\draw[mid>] (a0) -- node[right] {$k$} (a1);
\draw[mid<] (a1) -- node[right] {$n-k$} (a2);
\draw (a1) -- +(-0.2,0.1);
}
\qquad
\tikz[baseline=0.7cm]{
\foreach \n in {0,1,2} {
	\coordinate (a\n) at (0.4*\n, 0.8*\n);
}
\draw[mid<] (a0) -- node[right] {$k$} (a1);
\draw[mid>] (a1) -- node[right] {$n-k$} (a2);
\draw (a1) -- +(-0.2,0.1);
}
\end{align*}
with all labels drawn from the set $\{1,\ldots,n-1\}$. The third and fourth graphs depict bivalent vertices, called `tags', which are not rotationally symmetric, meaning that the tag provides a distinguished side. The bottom boundary of any planar graph in $\Hom(\ul{k}, \ul{l})$ is $\ul{k}$ with the strand oriented up for each positive entry, and the strand oriented down for each negative entry. Similarly, the top boundary is determined by $\ul{l}$ in the same way.

\begin{example}
We can build a trivalent vertex with one incoming edge labelled by $n-2$ and two outgoing edges labelled by $n-1$, for example as
\begin{equation}
\begin{tikzpicture}
\coordinate (z1) at (0,0);
\coordinate (z2) at (2,0);
\coordinate (c) at (1,1);
\coordinate (e) at (1,2);
\coordinate (ce) at (1,1.666);
\coordinate (cz1) at (0.333,0.333);
\coordinate (cz2) at (1.666,0.333);
\draw[mid<] (z1) node[below] {$n{-}1$} -- (cz1);
\draw[mid<] (z2) node[below] {$n{-}1$} -- (cz2);
\draw[mid<] (ce) -- (e) node[above] {$n{-}2$};
\draw[mid>] (cz1) --node[left]{$1$} (c);
\draw[mid>] (cz2) --node[right]{$1$} (c);
\draw[mid>] (c) --node[right]{$2$} (ce);
\draw(ce) -- + (0.2,0);
\draw(cz1) -- + (-0.15,0.15);
\draw(cz2) -- + (-0.15,-0.15);
\end{tikzpicture}
\end{equation}
In this example there are various choices about which direction each tag points.
Once we impose the relations in the spider category, these choices will all become equal, up to a sign, via Equation \eqref{eq:switch}.
\end{example}

\begin{example}
There are several ways to build a trivalent vertex with all edges oriented inwards. For instance,
\begin{equation}
\begin{tikzpicture}[baseline=20]
\coordinate (z1) at (0,0);
\coordinate (z2) at (1,0);
\coordinate (c) at (0.5,0.5);
\coordinate (ce) at (0.5,1);
\coordinate (e) at (0.5,1.45);
\draw[mid>] (z1) node[below] {$k$} -- (c);
\draw[mid>] (z2) node[below] {$l$} -- (c);
\draw[mid>] (c) -- node[right] {$k{+}l$} (ce);
\draw[mid<] (ce) -- (e) node[above] {$n{-}k{-}l$};
\draw (ce) -- +(0.2,0);
\end{tikzpicture}
\qquad \text{or} \qquad
\begin{tikzpicture}[baseline=20]
\coordinate (z1) at (0,0);
\coordinate (z2) at (1.5,0);
\coordinate (cze) at (1.25,0.333);
\coordinate (c) at (0.75,1);
\coordinate (e) at (0.75,1.45);
\draw[mid>] (z1) node[below] {$k$} -- (c);
\draw[mid>] (z2) node[below] {$l$} -- (cze);
\draw[mid<] (cze) -- node[right] {$n{-}l$} (c);
\draw (cze) -- + (-0.15,-0.15);
\draw[mid<] (c) -- (e) node[above] {$n{-}k{-}{l}$};
\end{tikzpicture}
\end{equation}
There are choices both in where around the trivalent vertex to place the tag, and on which side of the edge the tag lies. Again, these will all become equal (possibly up to a sign), via Equations \eqref{eq:switch} and \eqref{eq:tag-migration}.
\end{example}

We will often draw diagrams with edges also labelled by $0$ or $n$. This is a notational convenience, to be interpreted as follows. Edges labelled by $0$ and $n$ are to be deleted. Trivalent vertices involving a $0$ edge become simple strands and trivalent vertices involving an edge labelled by $n$ are replaced with tags:
\begin{align*}
\fuse{k}{n-k}{n} & = \tikz[baseline=0.5cm]{\draw[mid>] (0,0) node[below] {$k$} arc (180:90:0.6) node[coordinate] (c) {}; \draw[mid<] (c) arc (90:0:0.6) node[below] {$n-k$}; \draw (c) -- +(0,0.2);} &
\fork{k}{n-k}{n} & = \tikz[baseline=-0.5cm]{\draw[mid<] (0,0) node[above] {$k$} arc (-180:-90:0.6) node[coordinate] (c) {}; \draw[mid>] (c) arc (-90:0:0.6) node[above] {$n-k$}; \draw (c) -- +(0,0.2);}
\end{align*}
Any trivalent vertices with all edges labelled either $0$ or $n$ can be deleted. We will occasionally utilize diagrams with an edge labelled less than $0$ or greater than $n$; by convention these diagrams are 0.

\subsection{Definition of the spider category  \texorpdfstring{$\Sp(\SL_n)$}{Sp(SL\_n)} }\label{sec:spider}
The spider category $\Sp(\SL_n)$ is the quotient of $\FSp(\SL_n)$ by the following relations:

\begin{align}
\tikz[baseline=0.4cm]{
\foreach \n in {0,1,2} {
	\coordinate (a\n) at (0.4*\n, 0.8*\n);
}
\draw[mid>] (a0) -- node[right] {$k$} (a1);
\draw[mid<] (a1) -- node[right] {$n-k$} (a2);
\draw (a1) -- +(-0.2,0.1);
}
& = (-1)^{k(n-k)}
\tikz[baseline=0.4cm]{
\foreach \n in {0,1,2} {
	\coordinate (a\n) at (0.4*\n, 0.8*\n);
}
\draw[mid>] (a0) -- node[right] {$k$} (a1);
\draw[mid<] (a1) -- node[right] {$n-k$} (a2);
\draw (a1) -- +(0.2,-0.1);
}
\displaybreak[1]
\label{eq:switch}
\\
\begin{tikzpicture}[baseline=20]
\foreach \n in {0,...,3} {
	\coordinate (z\n) at (0.4*\n, 0.8*\n);
}
\draw[mid>] (z0) -- node[right] {$k+l$} (z1);
\draw[mid>] (z2) -- node[right] {$k+l$} (z3);
\draw[mid>] (z1) to[out=150,in=-190] node[left] {$k$} (z2);
\draw[mid>] (z1) to[out=-30,in=0] node[right] {$l$} (z2);
\end{tikzpicture}
& = \qBinomial{k+l}{l}
\tikz[baseline=20]{\draw[mid>] (0,0) -- node[right] {$k+l$} (1,2);}
\label{eq:bigon1}
\displaybreak[1] \\
\begin{tikzpicture}[baseline=20]
\foreach \n in {0,...,3} {
	\coordinate (z\n) at (0.4*\n, 0.8*\n);
}
\draw[mid>] (z0) -- node[right] {$k$} (z1);
\draw[mid>] (z2) -- node[right] {$k$} (z3);
\draw[mid<] (z1) to[out=150,in=-190] node[left] {$l$} (z2);
\draw[mid>] (z1) to[out=-30,in=0] node[right] {$k+l$} (z2);
\end{tikzpicture}
& = \qBinomial{n-k}{l}
\tikz[baseline=20]{\draw[mid>] (0,0) -- node[right] {$k$} (1,2);}
\label{eq:bigon2}
\displaybreak[1] \\
\begin{tikzpicture}[baseline]
\foreach \x/\y in {0/0,1/0,2/0,0/1,1/1,0/2} {
	\coordinate(z\x\y) at (\x+\y/2,\y/1.5);
}
\coordinate (z03) at (1,2);
\draw[mid>] (z00) node[below] {$k$} --  (z01);
\draw[mid>] (z01) -- node[left] {$k+l$} (z02);
\draw[mid>] (z10) node[below] {$l$} -- (z01);
\draw[mid>] (z20) node[below] {$m$} -- (z02);
\draw[mid>](z02) -- node[left] {$k+l+m$} (z03);
\end{tikzpicture}
& =
\begin{tikzpicture}[baseline]
\foreach \x/\y in {0/0,1/0,2/0,0/1,1/1,0/2} {
	\coordinate(z\x\y) at (\x+\y/2,\y/1.5);
}
\coordinate (z03) at (1,2);
\draw[mid>] (z00) node[below] {$k$} --  (z02);
\draw[mid>] (z10) node[below] {$l$} -- (z11);
\draw[mid>] (z20) node[below] {$m$} -- (z11);
\draw[mid>] (z11) -- node[right] {$l+m$} (z02);
\draw[mid>](z02) -- node[left] {$k+l+m$} (z03);
\end{tikzpicture}
\label{eq:IH}
\displaybreak[1] \\
\label{eq:tag-migration}
\begin{tikzpicture}[baseline=20]
\coordinate (z1) at (0,0);
\coordinate (z2) at (1,0);
\coordinate (c) at (0.5,0.5);
\coordinate (ce) at (0.5,1);
\coordinate (e) at (0.5,1.45);
\draw[mid>] (z1) node[below] {$k$} -- (c);
\draw[mid>] (z2) node[below] {$l$} -- (c);
\draw[mid>] (c) -- node[right] {$k{+}l$} (ce);
\draw[mid<] (ce) -- (e) node[above] {$n{-}k{-}l$};
\draw (ce) -- +(0.2,0);
\end{tikzpicture}
&=
\begin{tikzpicture}[baseline=20]
\coordinate (z1) at (0,0);
\coordinate (z2) at (1.5,0);
\coordinate (cze) at (1.25,0.333);
\coordinate (c) at (0.75,1);
\coordinate (e) at (0.75,1.45);
\draw[mid>] (z1) node[below] {$k$} -- (c);
\draw[mid>] (z2) node[below] {$l$} -- (cze);
\draw[mid<] (cze) -- node[right] {$n{-}l$} (c);
\draw (cze) -- + (0.15,0.15);
\draw[mid<] (c) -- (e) node[above] {$n{-}k{-}l$};
\end{tikzpicture} \\
\label{eq:tag-migration2}
\begin{tikzpicture}[baseline=20]
\coordinate (z1) at (0,0);
\coordinate (z2) at (1,0);
\coordinate (c) at (0.5,0.5);
\coordinate (ce) at (0.5,1);
\coordinate (e) at (0.5,1.45);
\draw[mid>] (z1) node[below] {$k{+}l$} -- (c);
\draw[mid<] (z2) node[below] {$k$} -- (c);
\draw[mid>] (c) -- node[right] {$l$} (ce);
\draw[mid<] (ce) -- (e) node[above] {$n{-}l$};
\draw (ce) -- +(-0.2,0);
\end{tikzpicture}
&=
\begin{tikzpicture}[baseline=20,x=-1cm]
\coordinate (z1) at (0,0);
\coordinate (z2) at (1.5,0);
\coordinate (cze) at (1.25,0.333);
\coordinate (c) at (0.75,1);
\coordinate (e) at (0.75,1.45);
\draw[mid<] (z1) node[below] {$k$} -- (c);
\draw[mid>] (z2) node[below] {$k{+}l$} -- (cze);
\draw[mid<] (cze) -- node[left] {$n{-}k{-}l$} (c);
\draw (cze) -- + (0.15,0.15);
\draw[mid<] (c) -- (e) node[above] {$n{-}l$};
\end{tikzpicture} \\
\label{eq:id1b}
\tikz[baseline=40]{
\laddercoordinates{1}{2}
\ladderEn{0}{0}{$k-s$}{$l+s$}{$s$}
\ladderEn{0}{1}{$k-s-r$}{$l+s+r$}{$r$}
\node[left] at (l00) {$k$};
\node[right] at (l10) {$l$};
}
&=
\qBinomial{r+s}{r}
\tikz[baseline=20]{
\laddercoordinates{1}{1}
\ladderEn{0}{0}{$k-s-r$}{$l+s+r$}{$r+s$}
\node[left] at (l00) {$k$};
\node[right] at (l10) {$l$};
}
\displaybreak[1]
\\
\label{eq:commutation}
\begin{tikzpicture}[baseline=40]
\laddercoordinates{1}{2}
\node[left] at (l00) {$k$};
\node[right] at (l10) {$l$};
\ladderEn{0}{0}{$k{-}s$}{$l{+}s$}{$s$}
\ladderFn{0}{1}{$k{-}s{+}r$}{$l{+}s{-}r$}{$r$}
\end{tikzpicture}
&= \sum_t \qBinomial{k-l+r-s}{t}
\begin{tikzpicture}[baseline=40]
\laddercoordinates{1}{2}
\node[left] at (l00) {$k$};
\node[right] at (l10) {$l$};
\ladderFn{0}{0}{$k{+}r{-}t$}{$l{-}r{+}t$}{$r{-}t$}
\ladderEn{0}{1}{$k{-}s{+}r$}{$l{+}s{-}r$}{$s{-}t$}
\end{tikzpicture}
\end{align}
together with the mirror reflections and the arrow reversals of these. These relations are often refered to as the `switching a tag' \eqref{eq:switch}, `removing a bigon' \eqref{eq:bigon1} and \eqref{eq:bigon2}, `$I=H$' (\ref{eq:IH}), `tag migration' (\ref{eq:tag-migration} and \ref{eq:tag-migration2}), `square removal' \eqref{eq:id1b} and `square switch' \eqref{eq:commutation}.

\begin{rem}
In the relations above we allow strands to be labelled by $0$ and $n$. As before, this means that $0$-strands should be deleted and $n$-strands replaced by tags. 
\end{rem}

\begin{rem}
The relations above are redundant. For example, relation (\ref{eq:id1b}) is not necessary, following readily from relations \eqref{eq:bigon1} and \eqref{eq:IH} (alternatively, \eqref{eq:bigon1} is a special case of \eqref{eq:id1b} with $l=k-s-r=0$).  Relation \eqref{eq:bigon2} is a special case of relation \eqref{eq:commutation}, with some edges labelled by 0 or $n$.  Relation \eqref{eq:tag-migration} is a special case of \eqref{eq:IH} with $ k+l+m=n$. Moreover, relation \eqref{eq:commutation} for $r,s > 1 $ follows from the square switch relation with $ r = s = 1 $ (and the rest of the relations).  There is an easy diagrammatic proof for these facts or they can be proven as consequences of our main theorem. We give the above over-complete list of relations because they would be needed if we worked over $ \mathbb{Z}[q,q^{-1}] $ rather than $ \bC(q) $.
\end{rem}

\begin{lem}\label{lem:consequences} The following are consequences of the relations above:
\begin{align}
\tikz[baseline]{\draw[->] (0,0.5) node[above] {$k$} arc (45:-315:0.5cm);}
&= \qBinomial{n}{k} \label{eq:loop} \\
\label{eq:cancel-tags}
\tikz[baseline=0.6cm]{
\foreach \n in {0,1,2,3} {
	\coordinate (a\n) at (0.4*\n, 0.8*\n);
}
\draw[mid>] (a0) -- node[right] {$k$} (a1);
\draw[mid<] (a1) -- node[right] {$n-k$} (a2);
\draw[mid>] (a2) -- node[right] {$k$} (a3);
\draw (a1) -- +(-0.2,0.1);
\draw (a2) -- +(0.2,-0.1);
}
&= \tikz[baseline=0.6cm]{\draw[mid>] (0,0) -- node[right] {$k$} (1.2,2.4);}
\end{align}
\begin{equation}
\label{eq:serre}
\renewcommand{\ladderY}{1}
\begin{ladder}{2}{3}
\ladderE{0}{0}{}{}
\ladderE{0}{1}{}{}
\ladderE{1}{2}{}{}
\ladderI{0}{2}
\ladderIn{2}{0}{2}
% \node[below] at (l00) {$a$};
%\node[below] at (l10) {$b$};
%\node[below] at (l20) {$c$};
\end{ladder}
- \qi{2}
\begin{ladder}{2}{3}
\ladderE{0}{0}{}{}
\ladderE{1}{1}{}{}
\ladderE{0}{2}{}{}
\ladderI{0}{1}
\ladderI{2}{0}
\ladderI{2}{2}
\end{ladder}
+
\begin{ladder}{2}{3}
\ladderE{1}{0}{}{}
\ladderE{0}{1}{}{}
\ladderE{0}{2}{}{}
\ladderI{0}{0}
\ladderIn{2}{1}{2}
\end{ladder}
= 0
\end{equation}
where we use the convention that any non-vertical unlabelled strand carries a $1$, while the vertical strands have arbitrary compatible labels.
\end{lem}
\begin{proof}
The first identity follows from relation \eqref{eq:bigon2} with $k=0$ after deleting the 0-strings.  The second also follows from \eqref{eq:bigon2} with $l=n-k$ after replacing the $n$-strand with a matching pair of tags.

{ \renewcommand{\ladderY}{1}
Finally, to prove (\ref{eq:serre}), we apply the $I=H$ relation along the leftmost upright to obtain
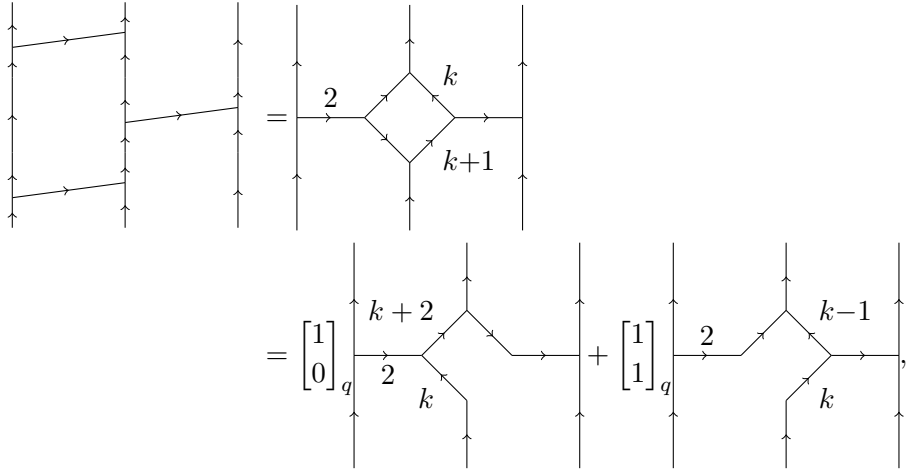
\begin{align*}
\begin{ladder}{2}{3}
\ladderE{0}{0}{}{}
\ladderE{1}{1}{}{}
\ladderE{0}{2}{}{}
\ladderI{0}{1}
\ladderI{2}{0}
\ladderI{2}{2}
\end{ladder}
& =
\begin{tikzpicture}[baseline=40]
\laddercoordinates{2}{3}
\coordinate (m0) at ($(l00)!.5!(l03)$);
\coordinate (m2) at ($(l20)!.5!(l23)$);
\coordinate (m1m) at ($(m0)!.3!(m2)$);
\coordinate (m1p) at ($(m0)!.7!(m2)$);
\coordinate (m1d) at ($(l10)!.3!(l13)$);
\coordinate (m1u) at ($(l10)!.7!(l13)$);
\draw[mid>] (l00) -- (m0);
\draw[mid>] (m0) -- (l03);
\draw[mid>] (l20) -- (m2);
\draw[mid>] (m2) -- (l23);
\draw[mid>] (l10) -- (m1d);
\draw[mid>] (m1u) -- (l13);
\draw[mid>] (m0) -- node[above] {2} (m1m);
\draw[mid>] (m1p) -- (m2);
\draw[mid>] (m1m) -- (m1u);
\draw[mid>] (m1m) -- (m1d);
\draw[mid>] (m1d) --node[below right] {$k{+}1$} (m1p);
\draw[mid>] (m1p) --node[above right] {$k$} (m1u);
\end{tikzpicture}
\\
& =
\qBinomial{1}{0}
\begin{tikzpicture}[baseline=40]
\laddercoordinates{2}{3}
\coordinate (m0) at ($(l00)!.5!(l03)$);
\coordinate (m2) at ($(l20)!.5!(l23)$);
\coordinate (m1m) at ($(m0)!.3!(m2)$);
\coordinate (m1p) at ($(m0)!.7!(m2)$);
\coordinate (m1d) at ($(l10)!.3!(l13)$);
\coordinate (m1u) at ($(l10)!.7!(l13)$);
\draw[mid>] (l00) -- (m0);
\draw[mid>] (m0) -- (l03);
\draw[mid>] (l20) -- (m2);
\draw[mid>] (m2) -- (l23);
\draw[mid>] (l10) -- (m1d);
\draw[mid>] (m1u) -- (l13);
\draw[mid>] (m0) -- node[below] {2} (m1m);
\draw[mid>] (m1p) -- (m2);
\draw[mid>] (m1m) --node[above left] {$k+2$} (m1u);
\draw[mid<] (m1m) --node[below left] {$k$} (m1d);
\draw[mid<] (m1p) -- (m1u);
\end{tikzpicture}
+
\qBinomial{1}{1}
\begin{tikzpicture}[baseline=40]
\laddercoordinates{2}{3}
\coordinate (m0) at ($(l00)!.5!(l03)$);
\coordinate (m2) at ($(l20)!.5!(l23)$);
\coordinate (m1m) at ($(m0)!.3!(m2)$);
\coordinate (m1p) at ($(m0)!.7!(m2)$);
\coordinate (m1d) at ($(l10)!.3!(l13)$);
\coordinate (m1u) at ($(l10)!.7!(l13)$);
\draw[mid>] (l00) -- (m0);
\draw[mid>] (m0) -- (l03);
\draw[mid>] (l20) -- (m2);
\draw[mid>] (m2) -- (l23);
\draw[mid>] (l10) -- (m1d);
\draw[mid>] (m1u) -- (l13);
\draw[mid>] (m0) -- node[above] {2} (m1m);
\draw[mid>] (m1p) -- (m2);
\draw[mid>] (m1m) --(m1u);
\draw[mid>] (m1d) --node[below right] {$k$} (m1p);
\draw[mid>] (m1p) --node[above right] {$k{-}1$} (m1u);
\end{tikzpicture},
\end{align*}
where, to get the second equality, we apply Equation \eqref{eq:commutation} to the central square with $(k,l,r,s) = (2,k,k,1)$. Both coefficients here are equal to $+1$. Finally an application of Equation \eqref{eq:id1b} on each $2$-strand gives the desired identity.
}
\end{proof}

\begin{rem}
Later, we will use Equation \eqref{eq:serre} in the proof of Proposition \ref{prop:psi}, where it will correspond to the quantum Serre relation \ref{rel:3} in $\dU(\gl_m)$.
\end{rem}

\begin{rem}
Many more local relations hold, as consequences of these. In particular \cite{0704.1503} described another classes of relations, the `Kekul\'{e}' relations:
\begin{equation*}
 \sum_{k=-\sumhat{b}}^{-\sumtah{a} + 1} (-1)^{j+k} \qBinomial{j+k-\max{b}}{j-\ssum{b}} \qBinomial{\min{a} + n -j -k}{\ssum{a} + n -1 -j} \mathfig{0.2}{octagon_jk} = 0,
\end{equation*}
for each $\ssum{b} \leq j \leq \ssum{a}+n-1$ (each edge label is the signed sum of the blue arrows on either side, $\sumtah{a} = \sum{a} - \min{a}$, and $\sumhat{b} = \sum{b}-\min{b}$).
We do not know a diagrammatic argument deriving the Kekul\'{e} relations from the relations presented here; nevertheless, such a derivation must exist, by our main theorem.  (For the simplest Kekul\'{e} relation, originally found by Kim \cite{math.QA/0310143}, we do have such a derivation.)

Further, the main theorem of this paper in particular implies that any closed spider diagram can be reduced to a scalar multiple of the empty diagram, by successive application of the given relations, but we do not have such an evaluation algorithm at this point.
\end{rem}

\section{Statement and proof of the main theorem}\label{sec:theorem}

Recall that $U_q(\sl_n) $ is a $ \bC(q)$-algebra with generators $ E_i, F_i, K_i $ for $ i = 1, \dots, n-1 $ and the following relations
\begin{align*}
& K_iK_j=K_jK_i, \ \ K_jE_iK_j^{-1} = q^{\la i,j \ra} E_i, \ \  K_jF_jK_j^{-1} = q^{-\la i,j \ra} F_i \\
& [E_i,F_j] = \delta_{ij} \frac{K_i - K_i^{-1}}{q-q^{-1}} \\
& [2]_q E_iE_jE_i = E_i^2E_j + E_jE_i^2 \text{ if } |i-j| = 1, \ \ \ [E_i,E_j]=0 \text{ if } |i-j|>1
\end{align*}
where
$$\la i,j \ra = \begin{cases} 2 & \text{ if } i = j \\ -1 & \text{ if } |i-j|=1 \\ 0 & \text{ otherwise. } \end{cases}$$
It is a Hopf algebra with the coproduct given by
\begin{equation} \label{eq:Hopf}
\Delta(E_i) = E_i \otimes K_i + 1 \otimes E_i, \ \ \  \Delta(F_i) = F_i \otimes 1 + K_i^{-1} \otimes F_i, \ \ \ \Delta(K_i) = K_i \otimes K_i
\end{equation}
the antipode by
\begin{equation}\label{antipode}
S(K_i) = K_i^{-1}, \ \ \, S(E_i) = -E_iK_i^{-1}, \ \ \ S(F_i) = -K_i F_i
\end{equation}
and the conunit by
\begin{equation}\label{counit}
\epsilon(K_i) = 1, \ \ \ \epsilon(E_i) = 0, \ \ \ \epsilon(F_i) = 0.
\end{equation}

We will study the category $\Rep(\SL_n)$ whose objects are representation of $U_q(\sl_n) $ isomorphic to tensor products of the fundamental representations $\Alt{k}{q} \bC_q^n$ of $U_q(\sl_n)$.

\subsection{Some generating morphisms}
\label{sec:generating-morphisms}

Denote by $ x_1, \dots, x_n $ the usual basis of the standard $U_q(\sl_n)$-module $\bC_q^n $. Note that $S^2_q \bC_q^n $ is spanned by
$$ x_i \otimes x_j + q x_j \otimes x_i, \text{ for }  i < j , \text{ and } x_i^2, \text{ for all } i. $$
We define the {\bf quantum exterior algebra} of $ \bC_q^n $
$$  \Alt{\bullet}{q}(\bC_q^n) := T \bC_q^n / \langle S^2_q ( \bC_q^n) \rangle $$
to be the tensor algebra, over $\bC(q) $, of $ \bC_q^n $ modulo the quantum symmetric square (see \cite{BZ}). The space $ \Alt{\bullet}{q}(\bC_q^n) $ is a graded $U_q(\sl_n)$-module algebra and we denote the product by $ \wedge_q $. Thus in $ \Alt{\bullet}{q}(\bC_q^n) $ we have that
$$ x_i \wedge_q x_j + q x_j \wedge_q x_i = 0 \text{ for }  i < j, \text{ and } x_i \wedge_q x_i = 0, \text{ for all }  i. $$
If $ S =\{k_1, \dots, k_a\} \subset \{1, \dots, n\} $, with $ k_1 > \dots > k_a $, we write $ x_S = x_{k_1} \wedge_q \cdots \wedge_q x_{k_a} \in \Alt{a}{q}(\bC_q^n) $. The set $ x_S $ where $ S $ ranges over $ k $ element subsets of $ \{1, \dots, n \} $ forms a basis for $ \Alt{k}{q}(\bC_q^n) $. Note that $E_i, F_i$ and $K_i$ act as follows:
$$E_i(x_k) = \begin{cases} x_{k-1} & \text{ if } i=k-1 \\ 0 & \text{ otherwise } \end{cases} \ \
F_i(x_k) = \begin{cases} x_{k+1} & \text{ if } i=k \\ 0 & \text{ otherwise } \end{cases} \ \
K_i(x_k) = \begin{cases} qx_k & \text{ if } k = i \\ q^{-1} x_k & \text{ if } k=i+1 \\ x_k & \text{ otherwise. } \end{cases}$$
Together with the comultiplication this determines how $E_i,F_i,K_i$ act on $x_S$.

We now define a generating set of morphisms in $\Rep(\SL_n)$. If $S, T$ are two disjoint subsets of $ \{1, \dots, n\} $ we define $$ \ell(S, T) = |\{ (i,j) : i \in S, j \in T \text{ and } i < j \}|. $$ Note that $\ell(S,T) + \ell(T,S) = |S||T| $. We define $ M_{k,l} : \Alt{k}{q} (\bC^n_q) \otimes \Alt{l}{q}(\bC^n_q) \rightarrow \Alt{k+l}{q}(\bC^n_q) $ to be the multiplication map $ \wedge_q $, so that we have
\begin{equation*}
M_{k,l}(x_S \otimes x_T) = x_S \wedge_q x_T = \begin{cases} (-q)^{\ell(S, T)} x_{S \cup T} & \text{ if } S \cap T = \emptyset \\
 0 & \text{ otherwise. }
 \end{cases}
\end{equation*}
Note that $ M_{k,l} $ is a $ U_q(\sl_n)$-module map by the definition of the quantum exterior algebra.

On the other hand, we define a $ \bC(q)$-linear map $M'_{k,l} :\Alt{k+l}{q}(\bC^n_q) \rightarrow \Alt{k}{q} (\bC^n_q) \otimes \Alt{l}{q}(\bC^n_q) $ as follows
\begin{align*}
M'_{k,l}(x_S) = (-1)^{kl} \sum_{T \subset S} (-q)^{-\ell(S \smallsetminus T, T)} x_T \otimes x_{S \smallsetminus T}
\end{align*}
where $ T $ ranges over $k$-element subsets of $ S $. Finally, we define $D_k : \Alt{k}{q}( \bC^n_q) \rightarrow (\Alt{n-k}{q}  (\bC^n_q))^*$ by
$$
 D_k(x_S)(x_T) = \begin{cases} (-q)^{\ell(S, T)} & \text{ if } S \cap T = \emptyset \\
 0 & \text{ otherwise. }
 \end{cases}
$$

\begin{lem}\label{lem:slnmaps}
The maps $M'_{k,l}$ and $D_k$ defined above are morphisms of $U_q(\sl_n)$-modules.
\end{lem}
\begin{proof}
We prove that $D_k$ is a map of $U_q(\sl_n)$-modules (the proof for $M'_{k,l}$ is similar). First we need to show that
\begin{equation}\label{eq:temp}
E_i(D_k(x_S))(x_T) = D_k(E_i(x_S))(x_T)
\end{equation}
for any $S,T \subset \{1,\dots,n\}$. On the one hand, if $S = \{k_1 > \dots > k_a\}$ then
$$E_i(x_S) = \begin{cases} x_{k_1} \wedge_q \dots \wedge_q x_{k_j-1} \wedge_q x_{k_{j+1}} \wedge_q \dots \wedge_q x_{k_a} & \text{ if } k_j = i+1 \\ 0 & \text{ otherwise. } \end{cases}$$
Thus, the right side of (\ref{eq:temp}) equals
$$\begin{cases} (-q)^{\ell(S',T)} & \text{ if } S' = T^c \\ 0 & \text{ otherwise } \end{cases}$$
where $S' := S \smallsetminus \{ i+1 \} \cup \{ i \}$ and $T^c = \{1, \dots, n\} \setminus T$ is the complement. On the other hand, using the antipode $S$, the left hand side of (\ref{eq:temp}) equals
\begin{align*}
D_k(x_S) (- E_iK_i^{-1}(x_T))
&= D_k(x_S) (-E_i(x_T)) \cdot \begin{cases} q^{-1} & \text{ if } i \in T, i+1 \not\in T \\ q & \text{ if } i+1 \in T, i \not\in T \\ 1 & \text{ otherwise } \end{cases} \\
&= \begin{cases} (-q)^{\ell(S,T')+1} & \text{ if } S^c = T' \\ 0 & \text{ otherwise } \end{cases}
\end{align*}
where $T' := T \smallsetminus \{i+1\} \cup \{ i \}$. Here we used that
$$K_i(x_T) = \begin{cases} q x_T & \text{ if } i \in T, i+1 \not\in T \\ q^{-1} x_T & \text{ if } i+1 \in T, i \not\in T \\ x_T & \text{ otherwise. } \end{cases}$$
Relation (\ref{eq:temp}) now follows since one can check that $S' = T^c$ if and only if $S' \cap T^c$ and if this is the case then $\ell(S',T) = \ell(S,T')+1$. The analogues of (\ref{eq:temp}) for $F_i$ and $K_i$ follow similarly.
\end{proof}

The following result is perhaps of independent interest.
\begin{lem} \label{le:coalg}
The space $\Alt{\bullet}{q}(\bC_q^n)$ carries the structure of a coassociative coalgebra, where comultiplication is given by the map
$$ \Delta = \bigoplus_{k+l=N} M'_{k,l}: \Alt{N}{q}(\bC^n_q) \rightarrow \bigoplus_{k+l=N} \Alt{k}{q}(\bC^n_q) \otimes \Alt{l}{q}(\bC^n_q)$$
and the counit $\epsilon: \Alt{\bullet}{q}(\bC_q^n) \rightarrow \bC$ by $x_S \mapsto \delta_{S,\emptyset}$.
\end{lem}
\begin{proof}
To check coassociativity, let $ A, B, C $ be disjoint and consider the coefficient of $ x_A \otimes x_B \otimes x_C $ in $ (\Delta \otimes 1) \circ \Delta (x_S)$.  Checking the definition we see that it equals
$$
(-1)^{|A \cup B||C| + |A||B|} (-q)^{- \ell(C, A \cup B) - \ell(B,A)} = (-1)^{|A||C| + |B||C| + |A||B|} (-q)^{- \ell(C, A) - \ell(C, B) - \ell(B,A)}
$$
which equals its coefficient  in $ (1 \otimes \Delta) \circ \Delta (x_S) $.

The counit identity follows from
\begin{equation*}
x_S \mapsto \sum_{k,l} (-1)^{kl} \sum_{T \subset S, |T|=k} (-q)^{-\ell(S \smallsetminus T,T)} x_T \otimes x_{S \smallsetminus T} \mapsto (-q)^{-\ell(S,\emptyset)} x_{\emptyset} \otimes x_S = x_S. \qedhere
\end{equation*}
\end{proof}

\subsection{Definition of the functor \texorpdfstring{$\Gamma_n: \Sp(\SL_n) \rightarrow \Rep(\SL_n)$}{Gamma\_n}} \label{sec:deffunctor}

We now define the functor from the spider to the representation category $ \Gamma_n:\Sp(\SL_n) \rightarrow \Rep(\SL_n) $. At the level of objects we take
$$(k_1^{\epsilon_1}, \dots, k_m^{\epsilon_m}) \mapsto \left(\Alt{k_1}{q} \bC_q^n\right)^{\epsilon_1} \otimes \dots \otimes \left(\Alt{k_m}{q} \bC_q^n\right)^{\epsilon_m} .$$
where $ \epsilon_i \in \{1, -1\} $ and we interpret $ W^{-1} $ as the dual representation $ W^* $.
For generating morphisms we take
$$ \fuse{k}{l}{k+l} \mapsto M_{k,l} \ \ \text{ and } \ \ \fork{k}{l}{k+l} \mapsto M'_{k,l}.$$

As a special case, this forces us to define $ \Gamma_n $ on tags by
$$
\tikz[baseline=0.4cm]{
\foreach \n in {0,1,2} {
	\coordinate (a\n) at (0.4*\n, 0.8*\n);
}
\draw[mid>] (a0) -- node[right] {$k$} (a1);
\draw[mid<] (a1) -- node[right] {$n-k$} (a2);
\draw (a1) -- +(-0.2,0.1);
} \mapsto D_k \ \ \text{ and } \ \
\tikz[baseline=0.4cm]{
\foreach \n in {0,1,2} {
	\coordinate (a\n) at (0.4*\n, 0.8*\n);
}
\draw[mid>] (a0) -- node[right] {$k$} (a1);
\draw[mid<] (a1) -- node[right] {$n-k$} (a2);
\draw (a1) -- +(0.2,-0.1);
} \mapsto (-1)^{k(n-k)} D_k.
$$

\begin{thm}\label{thm:gamma}
This defines a pivotal functor $\Gamma_n: \Sp(\SL_n) \rightarrow \Rep(\SL_n)$.
\end{thm}
\begin{proof}
By Lemma \ref{lem:slnmaps} we have a well defined map $\FSp(\SL_n) \rightarrow \Rep(\SL_n)$.  It remains to show that this maps factors through $\Sp(\SL_n)$, which means checking relations (\ref{eq:switch}) -- (\ref{eq:commutation}).

{\bf Relation (\ref{eq:bigon1}).} We need to compute $M_{k,l} \circ M'_{k,l} = \qBinomial{k+l}{k} {\rm{id}}$. For $S $ with $|S|=k+l$ we have
\begin{align*}
M_{k,l} \circ M'_{k,l}(x_S)
&= M_{k,l} \left( (-1)^{kl} \sum_{T \subset S} (-q)^{- \ell( S \smallsetminus T, T)} x_T \otimes x_{S \smallsetminus T} \right) \\
&= (-1)^{kl}\sum_{T \subset S} (-q)^{\ell(T, S \smallsetminus T)} (-q)^{- \ell( S \smallsetminus T, T)} x_S \displaybreak[1]\\
&= (-1)^{kl} (-q)^{-kl} \sum_{T \subset S} (-q)^{2\ell(T, S \smallsetminus T)} x_S \\
&= \qBinomial{k+l}{k} x_S
\end{align*}
using that $ \ell(T, S \smallsetminus T) + \ell( S \smallsetminus T, T) = |S \smallsetminus T||T| =kl$ and using Lemma \ref{lem:binom} to obtain the last equality.

{\bf Relation (\ref{eq:bigon2}).} The bigon on the left hand side of (\ref{eq:bigon2}) is the composition of a cup, two trivalent vertices and a cap.  Because $ \Gamma_n $ is defined as a pivotal functor it takes the cup and cap to the ``cup'' and ``cap'' in $ \Rep(\SL_n)$.  We will consider $ (\Alt{k}{q} \bC_q^n)^* $ to be the left dual of $\Alt{k}{q} \bC_q^n$ with dual basis $ x_T^* $ (here $ T $ ranges over $ k$-element subsets of $ \{1, \dots, n\} $).  Thus the ``cap'' map $$(\Alt{k}{q} \bC_q^n)^* \otimes \Alt{k}{q} \bC_q^n \rightarrow \bC(q)$$
is just given by $ x_T^* \otimes x_U \mapsto \delta_{T,U}$.

On the other hand the ``cup'' map is given by the canonical copairing followed by the inverse of the pivotal isomorphism.  The pivotal isomorphism in the category $ \Rep(\SL_n)$ is given by the element $ K_{2\rho} = K_1^{n-1} \dots K_{n-1}$.  Thus, we see that the ``cup'' map is given by
\begin{align*}
\bC(q) &\rightarrow (\Alt{k}{q} \bC_q^n)^* \otimes \Alt{k}{q} \bC_q^n \\
1 &\mapsto \sum_{|T| = k} q^{k(n-k) - 2\ell(T, T^c)} x_T^* \otimes x_T
\end{align*}
where we use that $ K_{2\rho}^{-1} x_T = q^{\ell(T^c, T) - \ell(T, T^c)} x_T= q^{k(n-k) - 2\ell(T, T^c)}x_T$.
 
Thus, for $S \subset \{1,\dots,n\}$ with $|S|=k$, the left hand side of (\ref{eq:bigon2}) acts on $x_S$ as follows:
\begin{align*}
x_S
&\mapsto q^{l(n-l)} \sum_{|T|=l} (-q)^{-2 \ell(T, T^c)} x_T^* \otimes x_T \otimes x_S \\
&\mapsto q^{l(n-l)} \sum_{|T|=l, S \cap T = \emptyset} (-q)^{\ell(T,S) - 2 \ell(T, T^c)} x_T^* \otimes x_{T \cup S} \displaybreak[1]\\
&\mapsto (-1)^{kl} q^{l(n-l)} \sum_{|T|=l, S \cap T = \emptyset} (-q)^{\ell(T,S) - 2 \ell(T, T^c)} \sum_{U \subset T \cup S} (-q)^{-\ell((T \cup S) \smallsetminus U,U)} x_T^* \otimes x_U \otimes x_{(T \cup S) \smallsetminus U} \\
&\mapsto (-1)^{kl} q^{l(n-l)} \sum_{|T|=l, S \cap T = \emptyset} (-q)^{\ell(T,S) - \ell(S,T) - 2 \ell(T, T^c) } x_S \displaybreak[1]\\
&= (-1)^{kl} q^{l(n-l)} \sum_{|T|=l, S \cap T = \emptyset} (-q)^{-\ell(T,S) - \ell(S,T) - 2 \ell(T,T^c \smallsetminus S)} x_S \displaybreak[1]\\
&= q^{l(n-l-k)} \sum_{T \subset S^c} q^{-2\ell(T, T^c \smallsetminus S)} x_S \\
&= \qBinomial{n-k}{l} x_S
\end{align*}
where we write $S^c$ and $T^c$ for the complements of $S$ and $T$ in $\{1, \dots, n\}$. The result follows.

{\bf Relation (\ref{eq:IH}).}  This follows immediately from the fact that $ \Alt{\bullet}{q}(\bC^n_q) $ forms an associative algebra (it is a quotient of a tensor algebra) and the arrow reversal follows from the fact that it forms a coassociative coalgebra (Lemma \ref{le:coalg}).

{\bf Relation \eqref{eq:tag-migration}.}
As this is a special case of \eqref{eq:IH}, this relation follows.

{\bf Relation \eqref{eq:tag-migration2}.}
Note that $ \Hom(\Alt{k+l}{q} \bC_q^n \otimes (\Alt{k}{q} \bC_q^n)^*, (\Alt{n-l}{q} \bC_q^n)^*) $ is 1-dimensional and that both sides of \eqref{eq:tag-migration2} define non-zero elements of this space.  Thus there exists a scalar $ c \in \bC(q) $ such that 
\begin{equation*}
\begin{tikzpicture}[baseline=8]
\coordinate (z1) at (0,0);
\coordinate (z2) at (1,0);
\coordinate (c) at (0.5,0.5);
\coordinate (ce) at (0.5,1);
\coordinate (e) at (0.5,1.45);
\draw[mid>] (z1) node[below] {$k{+}l$} -- (c);
\draw[mid<] (z2) node[below] {$k$} -- (c);
\draw[mid>] (c) -- node[right] {$l$} (ce);
\draw[mid<] (ce) -- (e) node[above] {$n{-}l$};
\draw (ce) -- +(-0.2,0);
\end{tikzpicture}
= c \cdot \,
\begin{tikzpicture}[baseline=8,x=-1cm]
\coordinate (z1) at (0,0);
\coordinate (z2) at (1.5,0);
\coordinate (cze) at (1.25,0.333);
\coordinate (c) at (0.75,1);
\coordinate (e) at (0.75,1.45);
\draw[mid<] (z1) node[below] {$k$} -- (c);
\draw[mid>] (z2) node[below] {$k{+}l$} -- (cze);
\draw[mid<] (cze) -- node[left] {$n{-}k{-}l$} (c);
\draw (cze) -- + (0.15,0.15);
\draw[mid<] (c) -- (e) node[above] {$n{-}l$};
\end{tikzpicture}
\end{equation*}
where (abusing notation) the above diagrams represent their images under $\Gamma_n$ in $ \Rep(\SL_n) $.  We wish to show that $ c = 1 $.

Now, we precompose both sides with an upward pointing trivalent vertex, to obtain
\begin{equation*}
\begin{tikzpicture}[baseline=8]
\coordinate (z1) at (0,0);
\coordinate (z2) at (1,0);
\coordinate (c) at (0.5,0.5);
\coordinate (ce) at (0.5,1);
\coordinate (e) at (0.5,1.45);
\coordinate (b) at (0.5,-0.5);
\coordinate (a) at (0.5,-1);
\draw[mid>] (a) -- node[right] {$l$} (b);
\draw[mid>] (b) to[out=150,in=-150] node[left] {$k{+}l$} (c);
\draw[mid<] (b) to[out=30,in=-30] node[right] {$k$} (c);
\draw[mid>] (c) -- node[right] {$l$} (ce);
\draw[mid<] (ce) -- (e) node[above] {$n{-}l$};
\draw (ce) -- +(-0.2,0);
\end{tikzpicture}
= c \cdot \,
\begin{tikzpicture}[baseline=8,x=-1cm]
\coordinate (z1) at (0,0);
\coordinate (z2) at (1.5,0);
\coordinate (cze) at (1,0.333);
\coordinate (c) at (0.5,1);
\coordinate (e) at (0.5,1.45);
\coordinate (b) at (0.5,-0.5);
\coordinate (a) at (0.5,-1);
\draw[mid>] (a) -- node[right] {$l$} (b);
\draw[mid<] (b) to[out=30,in=-30] node[right] {$k$} (c);
\draw[mid>] (b) to[out=150,in=-90] node[left] {$k{+}l$} (cze);
\draw[mid<] (cze) to[out=90,in=210] node[left] {$n{-}k{-}l$} (c);
\draw (cze) -- + (0.2,0);
\draw[mid<] (c) -- (e) node[above] {$n{-}l$};
\end{tikzpicture}
\end{equation*}
Now the left hand side can be simplified using the second bigon relation \eqref{eq:bigon2} (which has been already proven to hold in $\Rep(\SL_n)$) to obtain $ \qBinomial{n-l}{k} D_l$ (recall that $\Gamma_n $ takes a tag to $ D_l$).

The right hand side can be simplified using the first tag migration relation \eqref{eq:tag-migration} followed by the first bigon relation \eqref{eq:bigon1} to obtain $c \qBinomial{n-l}{k} D_l$.

Thus we conclude that $ \qBinomial{n-l}{k}D_l = c \qBinomial{n-l}{k}D_l $ and thus $ c = 1 $ as desired.

{\bf Relation (\ref{eq:commutation}).}
We will prove (\ref{eq:commutation}) in the case when $r=s=1$ (this suffices by the second remark in \S \ref{sec:spider}), which amounts to the following diagram.
\begin{equation}
\begin{tikzpicture}[baseline=40]
\laddercoordinates{1}{2}
\node[left] at (l00) {$k$};
\node[right] at (l10) {$l$};
\ladderEn{0}{0}{$k{-}1$}{$l{+}1$}{1}
\ladderFn{0}{1}{$k$}{$l$}{1}
\end{tikzpicture}
=
\begin{tikzpicture}[baseline=40]
\laddercoordinates{1}{2}
\node[left] at (l00) {$k$};
\node[right] at (l10) {$l$};
\ladderFn{0}{0}{$k{+}1$}{$l{-}1$}{1}
\ladderEn{0}{1}{$k$}{$l$}{1}
\end{tikzpicture}
+
\qi{k-l}
\begin{tikzpicture}[baseline=40]
\laddercoordinates{1}{2}
\ladderIn{0}{0}{2}
\ladderIn{1}{0}{2}
\node[left] at (l02) {$k$};
\node[right] at (l12) {$l$};
\end{tikzpicture}
\label{eq:commute}
\end{equation}
It suffices to check this on $x_S \otimes x_T$ for some arbitrary $S,T \subset \{1, \dots, n\}$ with $|S|=k$ and $|T|=l$. The left hand side of (\ref{eq:commute}) acts as follows:
\begin{align*}
x_S \otimes x_T
&\mapsto (-1)^{k-1} \sum_{r \in S} (-q)^{-\ell(r,S \smallsetminus r)} x_{S \smallsetminus r} \otimes x_r \otimes x_T \\
&\mapsto (-1)^{k-1} \sum_{r \in S \smallsetminus T} (-q)^{-\ell(r,S \smallsetminus r) + \ell(r,T)} x_{S \smallsetminus r} \otimes x_{T \cup r} \displaybreak[1]\\
&\mapsto (-1)^{k+l-1} \sum_{r \in S \smallsetminus T} \sum_{r' \in T \cup r} (-q)^{-\ell(r,S \smallsetminus r) + \ell(r,T) - \ell((T \cup r) \setminus r', r')} x_{S \smallsetminus r} \otimes x_{r'} \otimes x_{(T \cup r) \smallsetminus r'} \\
&\mapsto (-1)^{k+l-1} \sum_{r \in S \smallsetminus T} \sum_{r' \in (T \smallsetminus S) \cup r} (-q)^{-\ell(r,S \smallsetminus r) + \ell(r,T) - \ell((T \cup r) \setminus r', r') + \ell(S \smallsetminus r, r')} x_{(S \smallsetminus r) \cup r'} \otimes x_{(T \cup r) \smallsetminus r'}.
\end{align*}
We can rewrite this sum depending on whether $r=r'$ or $r \ne r'$ in the latter case we get
\begin{equation}\label{eq:local}
(-1)^{k+l-1} \sum_{r \in S \smallsetminus T} \sum_{r' \in T \smallsetminus S} (-q)^{-\ell(r,S \smallsetminus r) + \ell(r,T) - \ell(T \smallsetminus r', r') + \ell(S,r')} x_{(S \smallsetminus r) \cup r'} \otimes x_{(T \cup r) \smallsetminus r'}.
\end{equation}
where, for convenience, we assumed $r > r'$. In the former case we get
\begin{align*}
& (-1)^{k+l-1} \sum_{r \in S \smallsetminus T} (-q)^{-\ell(r,S \smallsetminus r) + \ell(r,T) - \ell(T,r) + \ell(S \smallsetminus r,r)} x_S \otimes x_T \\
&= (-1)^{k+l-1} \sum_{r \in S \smallsetminus T} (-q)^{-2\ell(r,S \smallsetminus r) + 2\ell(r,T) - l + k - 1} x_S \otimes x_T \\
& = q^{-l+k-1} \sum_{r \in S \smallsetminus T} q^{2 (\ell(r,T \smallsetminus S) -  \ell(r, S \smallsetminus (T \sqcup \{r\}))) } x_S \otimes x_T
\end{align*}

On the other hand, the first term on the right hand side of (\ref{eq:commute}) acts as
\begin{align*}
x_S \otimes x_T
&\mapsto (-1)^{l-1} \sum_{r' \in T} (-q)^{-\ell(T \smallsetminus r', r')} x_S \otimes x_{r'} \otimes x_{T \smallsetminus r'} \\
&\mapsto (-1)^{l-1} \sum_{r' \in T \smallsetminus S} (-q)^{-\ell(T \smallsetminus r'), r') + \ell(S,r')} x_{S \cup r'} \otimes x_{T \smallsetminus r'} \\
&\mapsto (-1)^{k+l-1} \sum_{r' \in T \smallsetminus S} \sum_{r \in (S \smallsetminus T) \cup r'} (-q)^{\scalebox{0.7}{$-\ell(T {\smallsetminus} r', r') {+} \ell(S,r') {-} \ell(r, (S {\cup} r') {\smallsetminus} r)$}} x_{(S \cup r') \smallsetminus r} \otimes x_r \otimes x_{T \smallsetminus r'} \\
&\mapsto (-1)^{k+l-1} \sum_{r' \in T \smallsetminus S} \sum_{r \in (S \smallsetminus T) \cup r'} (-q)^{\scalebox{0.7}{$-\ell(T {\smallsetminus} r', r') {+} \ell(S,r') {-} \ell(r, (S {\cup} r') {\smallsetminus} r) {+} \ell(r, T {\smallsetminus} r')$}} x_{(S \cup r') \smallsetminus r} \otimes x_{(T \smallsetminus r') \cup r}
\end{align*}
Again, we have two cases, depending on whether $r = r'$ or $r \ne r'$. In the latter case we get the same expression as in (\ref{eq:local}). In the former case we end up with
\begin{align*}
& (-1)^{k+l-1} \sum_{r \in T \smallsetminus S} (-q)^{-\ell(T \smallsetminus r,r) + \ell(S,r) - \ell(r,S) + \ell(r,T \smallsetminus r)} x_S \otimes x_T \\
& \quad = q^{k-l-1} \sum_{r \in T \smallsetminus S} q^{2( - \ell(r,S) + \ell(r, T \smallsetminus r) +1) } x_S \otimes x_T
\end{align*}

Thus it suffices to prove the following: for any $S,T$ of size $k, l $ respectively,
\begin{equation} \label{toprove}
 \sum_{r \in S} q^{2(\ell(r, T) - \ell(r, S \smallsetminus r))} - \sum_{r \in T} q^{2 (-\ell(r, S) + \ell(r, T \smallsetminus r) + 1)} = q^{l-k+1} [k-l]_q.
\end{equation}
Since $\ell(r, T \smallsetminus r) = \ell(r, (T \smallsetminus S) \smallsetminus r) + \ell(r, S \cap T)$ and $\ell(r,S) = \ell(r, S \smallsetminus T) + \ell(r, S \cap T)$ we find that
$$\ell(r, T \smallsetminus r) - \ell(r,S) = \ell(r, (T \smallsetminus S) \smallsetminus r) - \ell(r, S \smallsetminus T).$$
Thus, in proving (\ref{toprove}) we can assume that $S$ and $T$ are disjoint.

We proceed by induction on $\min(k,l)$.  The base case of our induction will be when $ k = 0 $ or $ l = 0 $.  In this case, it is easy to see that \eqref{toprove} holds. Now assume that $k,l > 0 $.  Consider the elements of $ S \cup T $ arranged in order.  We can find a pair of consecutive entries one from $ S $ and one from $ T $.  More precisely, there exists $ s \in S $ and $ t \in T $ such that no element of $ S $ or $ T $ lies in between $ s $ and $ t $.   Suppose that $ s < t$ (if $ s> t$ then the argument is similar).  Then
$$ \ell(s,T) = \ell(t, T \smallsetminus t) + 1 \ \ \ \text{ and } \ \ \  \ell(s, S \smallsetminus s) = \ell(t, S) $$
and thus the left hand side of \eqref{toprove} is unchanged by the removal of $ s $ from $ S $ and $ t $ from $ T $ and so by induction \eqref{toprove} holds.

Finally, relation (\ref{eq:switch}) is straightforward while the proof of relation (\ref{eq:id1b}) is similar to that of (\ref{eq:bigon1}) above and we omit it. 
\end{proof}

\begin{lem}\label{lem:binom}
For any $n \ge k \in {\mathbb N}$ we have
\begin{equation}\label{eq:binom}
\qBinomial{n}{k} = q^{-k(n-k)} \sum_{S \subset T, |S|=k} q^{2 \ell(S, T \smallsetminus S)} = q^{k(n-k)} \sum_{S \subset T, |S|=k} q^{-2 \ell(S, T \smallsetminus S)}
\end{equation}
where $T = \{1, \dots, n\}$.
\end{lem}
\begin{proof}
We prove the first equality as the second follows in the same way. The proof is by induction on $n$. It is an elementary exercise to show that the left hand side of \eqref{eq:binom} satisfies the recursion relation
$$\qBinomial{n}{k} = q^k \qBinomial{n-1}{k} + q^{-n+k} \qBinomial{n-1}{k-1}.$$
It remains to show that the right hand side of (\ref{eq:binom}) also satisfies this recursion. To do this we break up the sum into two depending on whether $n \in S$. We have
$$q^{-k(n-k)} \sum_{\substack{S \subset T \\ |S|=k, n \not\in S}} q^{2 \ell(S, T \smallsetminus S)} = q^{-k(n-k)} q^{2k} \sum_{\substack{S \subset T \smallsetminus n \\ |S|=k} q^{2\ell(S,(T\setminus n)\setminus S)}} = q^k \qBinomial{n-1}{k}$$
where the last equality follows by induction. Similarly one finds that
$$q^{-k(n-k)} \sum_{\substack{S \subset T \\ |S|=k, n \in S}} q^{2 \ell(S, T \smallsetminus S)} = q^{-k(n-k)} \sum_{\substack{S \smallsetminus n \subset T \smallsetminus n \\ |S \smallsetminus n| = k-1}} q^{2 \ell(S \smallsetminus n, T \smallsetminus S)} = q^{-n+k} \qBinomial{n-1}{k-1}$$
The result follows by induction.
\end{proof}

\subsection{The main result}\label{sec:main}

\begin{thm}\label{thm:main}
The functor $\Gamma_n: \Sp(\SL_n) \rightarrow \Rep(\SL_n)$ is an equivalence of pivotal categories.
\end{thm}
\begin{proof}
We will use the following commutative diagram
\begin{equation}\label{diag:main}
\xymatrix{
\Lad_m^n \ar[r] \ar[d] & \dU^n(\gl_m) \ar[dr]^{\Phi_m^n} \ar[d]_{\Psi_m^n} & \\
\FSp(\SL_n) \ar[r] & \Sp(\SL_n) \ar[r]^{\Gamma_n} & \Rep(\SL_n) \\
}
\end{equation}
where the three categories in the bottom row were defined in section \ref{sec:diagrams}, $\Phi$ and $\dU^n(\gl_m)$ are defined in section \ref{sec:phi} while $\Lad_m^n$ and $\Psi$ are defined in section \ref{sec:ladders}.

We now explain why $\Gamma_n$ is an equivalence of categories. Since it is clearly an isomorphism on objects we must show that it is fully faithful.

Surjectivity (fullness) of $\Gamma_n$ on $\Hom$ spaces follows from the fullness of the functor $ \Phi^n_m $, which is proven in Theorem \ref{th:functorfullyfaithful}\footnote{The fullness of $\Gamma_n$ was proven in Proposition 3.5.8 of \cite{0704.1503} using Schur-Weyl duality instead of skew Howe duality, but the argument is essentially the same.}.  More precisely, given any two objects $ V, W $ in $\Rep(\SL_n) $ we can find some $m$ such that there exist $ n$-bounded weights $ \ul{k}, \ul{l} $ of $ U_q(\gl_m)$ such that $\Phi^n_m(\ul{k}) = V$ and $\Phi^n_m(\ul{l}) = W $. The fullness of $ \Phi^n_m $ tells us that the map
$$ \Phi^n_m : \one_{\ul{l}} \dU(\gl_m) \one_{\ul{k}} \rightarrow \Hom_{U_q(\sl_n)}(V, W) $$
is surjective (i.e. all the morphisms come from ladders with $m$ uprights). The commutativity of the right triangle (established in Proposition \ref{prop:commutes}) shows us that these morphisms all come from webs in $ \Sp(\SL_n) $.

Next we show that $ \Gamma_n $ is injective (faithful) on $ \Hom$ spaces. It suffices to do this on the $\Hom$ spaces between objects of the form $(k_1^+,\ldots,k_m^+)$ (that is, objects which are all oriented upwards), because every object is isomorphic (via a morphism built solely out of tags) to such an object.
Let $w$ be a morphism in $ \Sp(\SL_n)$ between upwards oriented objects such that $\Gamma_n(w)=0$.  By Theorem \ref{thm:laddering} and the commutativity of the left square (immediate from the definition of $\Psi$ in Proposition \ref{prop:psi}), we can find some $ m $ and some $\tilde{w} \in \dU^n(\gl_m) $ such that $ \Psi^n_m(\tilde{w}) = w $ by finding a ladder $ \tilde{w} $ equivalent to the web $ w $. Then by the commutativity of the right triangle, we see that $ \Phi^n_m(\tilde{w}) = 0 $.  However, by Theorem \ref{th:functorfullyfaithful}, $\Phi_m^n$ is faithful which means $\tilde{w}=0$ and hence $w=0$ as desired.
\end{proof}

\section{The functor \texorpdfstring{$\Phi_m^n:\dU(\gl_m) \rightarrow \Rep(\SL_n)$}{Phi}}\label{sec:phi}

\subsection{\texorpdfstring{$ U_q(\gl_m)$}{U\_q gl\_m} and its idempotent form \texorpdfstring{$\dU(\gl_m)$}{}}\label{sec:idemform}

We begin with the definition of $ U_q(\gl_m)$.  It is defined much the same way as $ U_q(\sl_m) $, except that we enlarge the ``torus'' by having invertible group-like generators $ L_1, \dots, L_n $ with $ K_i = L_i L_{i+1}^{-1} $.  In this way the weight spaces of $ U_q(\gl_m) $ are labelled by $ \bZ^m $.

We will also use Lusztig's idempotent form $ \dU(\gl_m) $. We regard $\dU(\gl_m)$ as a $\bC(q)$-linear category with objects $ \ul{k} = (k_1, \dots, k_m) \in \mathbb Z^m $.  The identity morphism of the object $\ul{k}$ is denoted $\one_\ul{k}$  and we write $ \one_\ul{l} \dU(\gl_m) \one_\ul{k}$ for the space of morphisms.

The morphisms are generated by $E_i^{(r)} \one_{\ul{k}} \in \one_{\ul{k} + r \alpha_i} \dU(\gl_m) \one_{\ul{k}} $ and $ F_i^{(r)} \one_{\ul{k}} \in \one_{\ul{k} - r \alpha_i} \dU(\gl_m) \one_{\ul{k}}$, for $i=1, \dots, m-1$ and $r \in {\mathbb N}$ (here $\alpha_i = (0,\dots,0,1,-1,0,\dots,0)$ where the $1$ appears in position $i$). Notice that $\Hom(\ul{k}, \ul{l}) = 0 $ unless $\sum k_i = \sum l_i$. When the specific weight space is not important (or is obvious from the context) we will write $E_i$ instead of $E_i \one_{\ul{k}}$, $F_i$ instead of $F_i \one_{\ul{k}}$ etc.

These morphisms satisfy the following set of relations:
\begin{align}
\label{rel:1}
E_i^{(r)} F_i^{(s)} \one_{\ul{k}} &= \sum_t \qBinomial{\la \ul{k}, \alpha_i \ra + r - s}{t} F_i^{(s-t)} E_i^{(r-t)} \one_{\ul{k}} & & \\
% \label{rel:1} E_i F_i \one_{\ul{k}} - F_i E_i \one_{\ul{k}} &= [\la \ul{k}, \alpha_i \ra]_q \one_{\ul{k}} & &\\
\label{rel:2}  E_i^{(r)} F_j^{(s)} \one_{\ul{k}} &= F_j^{(s)} E_i^{(r)} \one_{\ul{k}}, & & \text{ if $i \ne j$} \displaybreak[1]\\
\label{rel:3} E_iE_jE_i \one_{\ul{k}} &= (E_i^{(2)} E_j + E_j E_i^{(2)}) \one_{\ul{k}}, && \text{ if $ |i - j| = 1 $, and likewise with $F$'s,} \displaybreak[1]\\
\label{rel:4} E_i^{(r)} E_j^{(s)} \one_{\ul{k}} &= E_j^{(s)} E_i^{(r)} \one_{\ul{k}} &&\text{ if $ |i-j| > 1$, and likewise with $F$'s,} \\
\label{rel:5} E_i^{(s)} E_i^{(r)} &= \qBinomial{r+s}{r} E_i^{(r+s)} && \text{ and likewise with $F$'s.}
\end{align}
Here $\la \cdot, \cdot \ra$ is the standard inner product on $\mathbb Z^m$.

\begin{rem}
Since we work over $ \bC(q) $, we do not need the generators $ E_i^{(r)}, F_i^{(r)} $ for $ r > 1 $ as
$$E_i^{(r)} = \frac{E_i^r}{[r]_q \dots [1]_q} \ \ \text{ and } \ \ F_i^{(r)} = \frac{F_i^r}{[r]_q \dots [1]_q}.$$
We decided to list these extra generators since they appear naturally from the webs perspective.  Moreover, these extra generators are needed for Lusztig's $\mathbb{Z}[q,q^{-1}] $ form of the quantum group, though additional ``Serre-like'' relations (similar to \ref{rel:3}) are needed in that setting.
\end{rem}

A representation $ V $ of $ U_q(\gl_m) $ where the $ L_i $ act semisimply with all eigenvalues powers of $ q $ is equivalent to a functor from $ \dU(\gl_m) $ to the category of vector spaces which takes the object $ \ul{k} $ to the weight space $ V_{\ul{k}} := \{ v \in V : L_i v = q^{k_i} v \text{ for all } i \} $.

We will be interested in a certain truncation of $ \dU(\gl_m) $.  We say that a weight $ \ul{k} $ is an $n$-\textbf{bounded} if $ 0 \le k_i \le n $ for all $ i$.  We denote by $\dU^n(\gl_m)$ the quotient of $\dU(\gl_m)$ where we set to zero all objects which are not $n$-bounded. In other words, we quotient by the 2-sided ideal of morphisms generated by all $ \one_{\ul{k}} $ such that $ \ul{k} $ is not $ n$-bounded.

\subsection{Quantum skew Howe duality}\label{sec:quantumskew}

The vector space $\Alt{\bullet}{}(\mathbb C^n \otimes \mathbb C^m)$ carries commuting actions of $U(\sl_n)$ and $U(\gl_m)$.

\begin{thm}\mbox{}
The usual skew Howe duality \cite{MR986027,MR1321638} can be summarized as follows.
\begin{enumerate}
\item There is an isomorphism of $ U(\sl_n) $ representations
\begin{equation}
 \Alt{\bullet}{}(\mathbb C^n \otimes \mathbb C^m) \cong \left( \Alt{\bullet}{} \bC^n \right)^{\otimes m}
 \end{equation}
under which the $ \ul{k} $ weight space for the action of $ U(\gl_m) $ on the left hand side is identified with $\Alt{k_1}{} \mathbb C^n \otimes \cdots \otimes \Alt{k_m}{} \mathbb C^n$.
\item For each $ K $, the actions of $ U(\gl_m) $ and $ U(\sl_n) $ on $ \Alt{K}{}(\bC^n \otimes \bC^m) $ generate each other's commutant.
\item As a representation of $ U(\gl_m) \otimes U(\sl_n)$, we have a decomposition
$$ \Alt{\bullet}{} (\mathbb C^n \otimes \mathbb C^m) = \bigoplus_{\mu} V(\mu^t) \otimes V(\mu) $$
where $\mu$ varies over all $n$-bounded weights of $U(\gl_m)$. Here $\mu^t$ is the transpose of $\mu$, regarded as a weight of $U(\sl_n)$.
\end{enumerate}
\end{thm}

We will need to generalize this result to the quantum setting.  Unfortunately, there is not much literature concerning quantum skew Howe duality, so we will develop the theory here, following the ideas of Berenstein-Zwicknagl \cite{BZ}.

We consider $ \bC^n_q \otimes \bC^m_q $ as a representation of $U_q(\sl_n) \otimes U_q(\gl_m) = U_q(\sl_n \oplus \gl_m) $. Let us write $ x_1, \dots, x_n $ for the standard basis of $ \bC_q^n $ and $ y_1, \dots, y_m $ for the standard basis of $ \bC_q^m $.  Then $ \bC_q^n \otimes \bC_q^m $ has a basis given by $ z_{ij} := x_i \otimes y_j $.

We define the quantum exterior algebra of this representation to be the quotient of its tensor algebra by the ideal generated by its quantum symmetric square,
$$\Alt{\bullet}{q}(\bC_q^n \otimes \bC_q^m) := T (\bC_q^n \otimes \bC_q^m) / \langle S^2_q (\bC_q^n \otimes \bC_q^m) \rangle$$
Following the proof of \cite[Prop. 2.33]{BZ}, we have that
$$ S^2_q (\bC_q^n \otimes \bC_q^m) = (S^2_q \bC_q^n \otimes S^2_q \bC_q^m) \oplus (\Alt{2}{q} \bC_q^n \otimes \Alt{2}{q} \bC_q^m). $$
Continuing to follow the proof of \cite[Prop. 2.33]{BZ}, we see that $ S^2_q (\bC_q^n \otimes \bC_q^m) $ is spanned by
\begin{gather*}
(x_i \otimes x_i) \otimes (y_l \otimes y_l) \\
(x_i \otimes x_i) \otimes (y_l \otimes y_p + q y_p \otimes y_l), \text{ for } l < p \displaybreak[1]\\
(x_i \otimes x_j + q x_j \otimes x_i) \otimes (y_l \otimes y_l), \text{ for } i < j \displaybreak[1]\\
(x_i \otimes x_j + q x_j \otimes x_i) \otimes (y_l \otimes y_p + q y_p \otimes y_l), \text{ for } i < j, l < p \\
(q x_i \otimes x_j -  x_j \otimes x_i) \otimes (q y_l \otimes y_p -  y_p \otimes y_l), \text{ for } i < j, l < p.
\end{gather*}
A little manipulation proves that $ \Alt{\bullet}{q}(\bC_q^n \otimes \bC_q^m) $ is the quotient of the free algebra on the set $ \{ z_{ij} \} $ modulo the relations
\begin{align*}
z_{ij} \wedge_q z_{ij} &= 0 \\
z_{ij} \wedge_q z_{lj} &= - q z_{lj} \wedge_q z_{ij}  \text{ if } i < l \displaybreak[1]\\
z_{ij} \wedge_q z_{ip} &= - q z_{ip} \wedge_q z_{ij} \text{ if } j < p \displaybreak[1]\\
z_{ij} \wedge_q z_{lp} &= - z_{lp} \wedge_q z_{ij} \text{ if  } i < l, j < p \\
z_{ij} \wedge_q z_{lp} &= - z_{lp} \wedge_q z_{ij} + (q - q^{-1}) z_{ip} \wedge_q z_{lj} \text{ if } i < l, j > p
\end{align*}

From the general theory from \cite{BZ}, we see that the algebra $ \Alt{\bullet}{q}(\bC_q^n \otimes \bC_q^m) $ carries commuting actions of $ U_q(\sl_n) $ and $ U_q(\gl_m) $ (equivalently it carries an action of the quantum group $ U_q(\sl_n \oplus \gl_m) $).  The generators of $E_p, F_p, L_p \in U_q(\gl_m) $ act on the generators $ z_{ij} $ of $\Alt{\bullet}{q}(\bC_q^n \otimes \bC_q^m) $ in the obvious fashion
$$ E_p z_{ij} = \begin{cases} z_{i,j-1} & \text{ if } p = j-1  \\
 0 & \text{  otherwise }
 \end{cases}, \ \
F_p z_{ij} = \begin{cases} z_{i,j+1} & \text{ if } p = j \\
 0 & \text{ otherwise }
 \end{cases}, \ \
L_p z_{ij} = \begin{cases} q z_{ij} & \text{ if } p = j \\
z_{ij} & \text{ otherwise }
\end{cases}$$
and similarly for the generators of $ U_q(\sl_n) $.

Recall from \cite{BZ} that if $ V $ is a representation of a quantum group $ U_q(\mathfrak{g}) $, then $ \Alt{\bullet}{q}(V) $ always admits a $ q=1 $ specialization, denoted $ \overline{\Alt{\bullet}{q}(V)}$, which will be a quotient of $ \Alt{\bullet}{} \overline{V} $ (as a $U(\mathfrak{g})$-module). For certain special $ V $, we actually specialize to the entire exterior algebra --- this is true in our case.

\begin{thm} \label{th:qSkewHowe}\mbox{}
\begin{enumerate}
\item The specialization $\overline{\Alt{\bullet}{q}(\bC_q^n \otimes \bC_q^m)} $ is isomorphic, as a $U(\gl_m) \otimes U(\sl_n)$-module, to $ \Alt{\bullet}{}(\bC^n \otimes \bC^m) $.
\item For each $ K $, the actions of $ U_q(\gl_m) $ and $ U_q(\sl_n) $ on $ \Alt{K}{q}(\bC_q^n \otimes \bC_q^m) $ generate each other's commutant.
\item As a representation of $ U_q(\gl_m) \otimes U_q(\sl_n)  $, we have a decomposition
$$ \Alt{\bullet}{q} (\mathbb C^n_q \otimes \mathbb C^m_q) = \bigoplus_{\mu} V(\mu^t) \otimes V(\mu) $$
 where $\mu$ varies over all $n$-bounded weights of $ \dU(\gl_m)$.
\item We have an isomorphism as $ U_q(\sl_n) $ representations
$ \Alt{\bullet}{q}(\mathbb C^n_q \otimes \mathbb C^m_q) \cong \Alt{\bullet}{q}(\bC_q^n)^{\otimes m} $.  Moreover under this isomorphism, the $ \ul{k} $ weight space for the action of $ U(\gl_m) $ on the left hand side is identified with $\Alt{k_1}{q} \mathbb C_q^n \otimes \cdots \otimes \Alt{k_m}{q} \mathbb C_q^n$.
\end{enumerate}
\end{thm}

In fact the last part of this theorem can be strengthened to an algebra isomorphism, but we will not need this here.

\begin{proof}
We begin with statement (1). It suffices to show that $ \Alt{\bullet}{q}(\bC_q^n \otimes \bC_q^m) $   has the correct graded dimension.  To prove this, note that $\Alt{\bullet}{q}(\bC_q^n \otimes \bC_q^m) $ is the quadratic dual of the more familiar quantum matrix algebra $ S_q^\bullet (\bC_q^n \otimes \bC_q^m) $.  By \cite[Prop. 2.33]{BZ}, this algebra is flat and thus Koszul by \cite[Prop. 2.28]{BZ}. By numerical Koszul duality,
$$ h\left(\Alt{\bullet}{q}( \bC_q^n \otimes \bC_q^m), t^{-1}\right) h(S^\bullet_q (\bC_q^n \otimes \bC_q^m), t) = 1 $$
where $ h(V^\bullet, t) = \sum_k \dim V^k t^k $ denotes graded dimension.

Since $ S^\bullet_q (\bC_q^n \otimes \bC_q^m) $ is flat, $ h(S^\bullet_q (\bC_q^n \otimes \bC_q^m), t) =  h(S^\bullet (\bC^n \otimes \bC^m), t) $ and thus  $ h(\Alt{\bullet}{q} (\bC_q^n \otimes \bC_q^m), t) =  h(\Alt{\bullet}{} (\bC^n \otimes \bC^m), t)$ as desired.

By statement (1), we know that $ \Alt{\bullet}{q}(\bC_q^n \otimes \bC_q^m) $ decomposes into irreducible $ U_q(\gl_m \oplus \sl_n) $ representations in the same manner as $ \Alt{\bullet}{}(\bC^n \otimes \bC^m) $ decomposes into $ U(\gl_m \oplus \sl_n)$-modules. This immediately implies statement (3) which in turn implies (2).

Now we consider statement (4). For each $ 1 \le j \le m$, we define an algebra map $T_j : \Alt{\bullet}{q}(\bC_q^n) \rightarrow \Alt{\bullet}{q} (\bC_q^n \otimes \bC_q^m) $ by taking generators $ x_i $ to $z_{ij}$.  This is well-defined as an algebra map because the relations in $ \Alt{\bullet}{q}(\bC_q^n) $ are taken to relations in $\Alt{\bullet}{q} (\bC_q^n \otimes \bC_q^m)$.  Moreover $T_j $ is a map of $ U_q(\sl_n) $ representations.  Let us write
$$z_{S,j} = T_j(x_S) = z_{k_1,j} \wedge_q z_{k_2,j} \wedge_q \dots \wedge_q z_{k_a,j}$$
where $ S = \{k_1 > \dots > k_a \} \subset \{ 1, \dots, n \} $. By multiplying together the $ T_j $, we define
\begin{align*}
T : \Alt{\bullet}{q}(\bC_q^n) \otimes \cdots \otimes \Alt{\bullet}{q}(\bC_q^n) &\rightarrow \Alt{\bullet}{q} (\bC_q^n \otimes \bC_q^m) \\
v_1 \otimes \cdots \otimes v_m &\mapsto T_1(v_1) \wedge_q \cdots \wedge_q T_m(v_m)
\end{align*}
Since the multiplication map on $\Alt{\bullet}{q} (\bC_q^n \otimes \bC_q^m)  $ is $U_q(\sl_n)$-equivariant, $ T $ is a map of $ U_q(\sl_n) $ representations. For $ S_1, \dots, S_m \subset \{1, \dots, n \} $, we consider $ T(x_{S_1} \otimes \cdots \otimes x_{S_m}) = z_{S_1,1} \wedge_q \cdots \wedge_q z_{S_m,m}$. If we let $ S_j $ range over all subsets, these clearly span $ \Alt{\bullet}{q} (\bC_q^n \otimes \bC_q^m) $.  This means that they form a basis since the number of such elements equals the dimension of $ \Alt{\bullet}{q} (\bC_q^n \otimes \bC_q^m) $.  Thus we see that the map $ T $ is an isomorphism since it takes a basis to a basis.
\end{proof}

\begin{lem} \label{lem:Eaction}
The action of $E_p,F_p \in U_q(\gl_m)$ on $\Alt{\bullet}{q} (\bC_q^n \otimes \bC_q^m)$ is given by
\begin{align*}
& E_p (z_{S_1,1} \wedge_q \cdots \wedge_q z_{S_m,m})  \\
&\quad= (-1)^{|S_{p+1}|+1} \sum_{r \in S_{p+1}} (-q)^{- \# \{s \in S_{p+1}: s < r \}} z_{S_1,1} \wedge_q \cdots \wedge_q z_{S_p,p} \wedge_q z_{r,p} \wedge_q z_{S_{p+1} \smallsetminus r,p+1} \wedge_q \cdots \wedge_q z_{S_m,m} \notag \\
& F_p (z_{S_1,1} \wedge_q \cdots \wedge_q z_{S_m,m}) \\
&\quad=(-1)^{|S_p|+1} \sum_{r \in S_p} (-q)^{- \#\{s \in S_p: s > r \}} z_{S_1,1} \wedge_q \cdots \wedge_q z_{S_p \smallsetminus r,p} \wedge_q z_{r,p+1} \wedge_q z_{S_{p+1},p+1} \wedge_q \cdots \wedge_q z_{S_m,m}. \notag
\end{align*}
\end{lem}
\begin{proof}
We prove the first assertion in the case $m=2$ and $p=1$ since it simplifies notation and the general case is the same (the second assertion follows similarly). Recall that $z_{S,i} = z_{k_1,i} \wedge_q \dots \wedge_q z_{k_a,i}$ where $S = \{k_1 > \dots > k_a\}$. Since we are taking products we need to use the Hopf algebra structure of $U_q(\gl_m)$ which is given in (\ref{eq:Hopf}). In particular, $\Delta(E_p) = E_p \otimes K_p + 1 \otimes E_p$. We also note that
$$K_p(z_{ij}) =
\begin{cases}
q z_{ij} & \text{ if } j=p \\
q^{-1} z_{ij} & \text{ if } j=p+1 \\
z_{ij} & \text{ otherwise.}
\end{cases}$$
Suppose $S_1 = \{k_1 > \dots > k_a\}$ and $S_2 = \{l_1 > \dots > l_b\}$. Then we have
\begin{align*}
& E_1(z_{S_1,1} \wedge_q z_{S_2,2}) \\
&\quad= z_{S_1,1} \wedge_q ( z_{l_1,1} \wedge_q K_1(z_{l_2,2}) \wedge_q \dots \wedge_q K_1(z_{l_b,2}) + z_{l_1,2} \wedge_q z_{l_2,1} \wedge_q K_1(z_{l_3,2}) \wedge_q \dots \wedge_q K_1(z_{l_b,2}) + \dots ) \\
&\quad= q^{-b} z_{S_1,1} \wedge_q ( q z_{l_1,1} \wedge_q z_{l_2,2} \wedge_q \dots \wedge_q z_{l_b,2} + q^2 z_{l_1,2} \wedge_q z_{l_2,1} \wedge_q z_{l_3,2} \wedge_q \dots \wedge_q z_{l_b,2} + \dots) \\
&\quad= q^{-b} z_{S_1,1} \wedge_q (q z_{l_1,1} \wedge_q z_{S_2 \smallsetminus l_1,2} - q^2 z_{l_2,1} \wedge z_{S_2 \smallsetminus l_2,2} + q^3 z_{l_3,1} \wedge_q z_{S_2 \smallsetminus l_3,2} - \dots ).
\end{align*}
The result follows after some simplification.
\end{proof}

\subsection{Definition of the functor \texorpdfstring{$\Phi_m$}{Phi}}

We will now use  quantum skew Howe duality to define the functor $\Phi_m$. By Theorem \ref{th:qSkewHowe}(4), we have a $U_q(\gl_m)$ action with the weight spaces $\Alt{k_1}{q} \mathbb C_q^n \otimes \cdots \otimes \Alt{k_m}{q} \mathbb C_q^n$ and commuting with the $U_q(\sl_n)$ action. Thus we get a map
\begin{equation}\label{eq:phimap}
\one_\ul{l} \dU(\gl_m) \one_\ul{k} \rightarrow \Hom_{U_q(\sl_n)} \left(\Alt{k_1}{q} \mathbb C^n \otimes \cdots \otimes \Alt{k_m}{q} \mathbb C^n, \Alt{l_1}{q} \mathbb C^n \otimes \cdots \otimes \Alt{l_m}{q} \mathbb C^n\right)
\end{equation}
for any two $n$-bounded weights $\ul{k}, \ul{l}$ with $\sum_i k_i = \sum_i l_i $.

Since the action of $U_q(\gl_m)$  on $\Alt{K}{q}(\mathbb C^n \otimes \mathbb C^m) $ generates the commutant of the $U_q(\sl_n)$ action, the map (\ref{eq:phimap}) is surjective.

Thus we may define a functor $\Phi_m: \dU(\gl_m) \rightarrow \Rep(\SL_n)$ as follows:
\begin{itemize}
\item On objects
$\ul{k} \mapsto
\begin{cases}
\Alt{k_1}{q} \mathbb C_q^n \otimes \cdots \otimes \Alt{k_m}{q} \mathbb C_q^n & \text{ if } \ul{k} \text{ is } n\text{-bounded } \\
0 & \text{ otherwise.}
\end{cases}$
\item On morphisms $ \Phi_m $ is given by (\ref{eq:phimap}).
\end{itemize}
Since (\ref{eq:phimap}) was surjective the functor $\Phi_m$ is full.

\subsection{Fully-faithfulness of \texorpdfstring{$ \Phi_m^n$}{Phi}}
\label{sec:fully-faithful}

Since all weights of $\Alt{K}{}(\mathbb{C}^n \otimes \mathbb{C}^m)$ are $ n$-bounded, the functor $\Phi_m: \dU(\gl_m)\rightarrow \Rep(\SL_n)$ factors through $\dU^n(\gl_m)$. We denote this induced functor $\Phi_m^n: \dU^n(\gl_m) \rightarrow \Rep(\SL_n)$.

\begin{thm}\label{th:functorfullyfaithful}
The functor $\Phi_m^n: \dU^n(\gl_m) \rightarrow \Rep(\SL_n)$ is fully faithful (meaning that it induces an isomorphisms between $\Hom$-spaces).
\end{thm}

To prove this result, we need a general fact about reductive Lie algebras (and quantum groups). This result was not previously known to us and we thank the referee for pointing out the reference \cite[Thm. 4.2]{MR1990659}. We include here a proof for completeness.

For simplicity, we state this result in the $ \gl_m $ case. We will need to consider the algebra version (instead of the category version) of Lusztig's idempotent form,
$$
\dalg(\gl_m) := \bigoplus_{\ul{k}, \ul{l}} \one_{\ul{l}} \dU(\gl_m) \one_{\ul{k}}.
$$
In a similar fashion we define $\dalg^n(\gl_m)$ as a quotient of $\dalg(\gl_m)$.

For any dominant weight $ \lambda $, let $V(\lambda)$ the corresponding highest weight representation of $ U_q(\gl_m)$.

We have the usual dominance order on dominant weights of $ \gl_m $ where $ \mu \le \lambda $ if $ \lambda - \mu $ is a sum of the simple roots $ \alpha_i $.  We extend this notion as follows.  We say that a dominant weight $ \lambda $ \textbf{dominates} a weight $ \nu $, if $ \nu $ lies in the Weyl group orbit of a dominant weight $ \mu \le \lambda $.

Let $ I_\lambda $ be the 2-sided ideal in $\dalg(\gl_m)$ generated by all $ 1_\nu $ such that $ \lambda $ does not dominate $ \nu $. If $ \mu $ is a dominant weight of $ \gl_m $ with $ \mu \le \lambda $, then for each $ \nu $ as above, $ \nu $ is not a weight of $V(\mu)$.  Thus $ I_\lambda $ acts trivially on $ V(\mu) $ and we get a representation $  \dalg(\gl_m)/I_\lambda \rightarrow \End{V(\mu)} $ .

\begin{lem}
For any dominant weight $ \lambda $, the map $ \dalg(\gl_m) / I_\lambda \rightarrow \bigoplus_{\mu \le \lambda} \End{V(\mu)}$ is an isomorphism (where the sum is over dominant $\mu$).
\end{lem}
\begin{proof}
First note that $\dalg(\gl_m)/ I_\lambda $ is finite-dimensional. By Wedderburn's theorem, it suffices to show that the category of finite-dimensional $\dalg(\gl_m)/I_\l$-modules is semisimple with simple objects the $V(\mu)$, for $\mu \le \lambda$.

Now a $\dalg(\gl_m)/I_\l$ module is the same thing as a $\dalg(\gl_m) $ module in which $I_\l$ acts trivially. Since the category of finite-dimensional $\dalg(\gl_m) $ modules is semisimple and the ones where $I_\l$ acts trivially are precisely the $ V(\mu) $ for $ \mu \le \lambda$, the result follows.
\end{proof}

\begin{example}
Suppose $m=2$ and $\l = (2,0)$. Then the weights dominated by $\l$ are $(2,0), (1,1), (0,2)$. Subsequently,  the morphisms in $\dot{U}_q(\gl_2)/I_\l$ are spanned by
$$\one_{(2,0)}, \one_{(1,1)}, \one_{(0,2)}, E \one_{(0,2)}, E \one_{(1,1)}, E^2 \one_{(0,2)}, F \one_{(2,0)}, F\one_{(1,1)}, F^2 \one_{(2,0)}, EF\one_{(1,1)} = FE\one_{(1,1)}.$$
On the other hand, the \emph{dominant} weights dominated by $(2,0)$ are $(2,0)$ and $(1,1)$. Note that $V((2,0)) = S^2 \bC^2 $, which is three-dimensional, and $V((1,1)) = \Alt{2}{} \bC^2 $, which is one-dimensional.  Thus $\bigoplus_{\mu \le \l} \End{V(\mu)} \cong \End{\bC^3} \oplus \End{\bC}$. Notice that these spaces have the same dimension (they are both 10-dimensional).
\end{example}

\begin{proof}
We now return to proving Theorem \ref{th:functorfullyfaithful}. Recall that by quantum skew Howe duality, we have a decomposition
$$ \Alt{\bullet}{q} (\bC_q^n \otimes \bC_q^m) = \bigoplus_{\mu} V(\mu^t) \otimes V(\mu) $$
as $U_q(\sl_n) \otimes U_q(\gl_m)$-representations, where $\mu$ varies over all $n$-bounded weights of $U_q(\gl_m)$. Thus for any $ 0 \le K \le mn $,
$$ \Hom_{U_q(\sl_n)} \left(\Alt{K}{q}(\bC_q^n \otimes \bC_q^m), \Alt{K}{q}(\bC_q^n \otimes \bC_q^m)\right) = \bigoplus_{\mu} \End{V(\mu)} $$
where $ \mu $ ranges over $ n$-bounded weights with $ \sum \mu_i = K $.  Note that these $\mu $ are exactly the set of dominant weights of $ \gl_m $ which satisfy $ \mu \le \lambda(K) $, where $ \lambda(K) $ is the unique weight of the form $(n, \dots, n, r, 0, \dots, 0) $ where the terms sum to $K$.  Applying the previous lemma, we see that the map
$$ \dalg(\gl_m)/I_{\lambda(K)} \rightarrow \Hom_{U_q(\sl_n)} \left(\Alt{K}{q}(\bC_q^n \otimes \bC_q^m), \Alt{K}{q}(\bC_q^n \otimes \bC_q^m)\right) $$
is an isomorphism. Since a weight $ \mu $ is $n$-bounded if and only if $ \mu $ is dominated by $ \lambda(K) $ where $ K = \sum \mu_i $ we get
$$ \dalg^n(\gl_m) = \bigoplus_{K=0}^{nm} \dalg(\gl_m)/I_{\lambda(K)} $$
and the result follows.
\end{proof}

\begin{rem}
When $ K = n $, the algebra $ \dalg(\gl_m)/I_{\lambda(K)} $ appearing above is known as the $q$-Schur algebra.  The algebras $ \dalg(\gl_m)/I_{\lambda(K)} $ for general $ K $ are called generalized $q$-Schur algebras by Doty \cite{MR1990659}.
\end{rem}

\section{Ladders}
\label{sec:ladders}

\subsection{Ladders and \texorpdfstring{$\dU^n(\gl_m)$}{U\_q gl\_m} }
We will now introduce a diagrammatic notation for morphisms in $ \dU^n(\gl_m)$. We begin by formalizing the notion of a ladder.

\begin{defn}
An {\bf $n$-ladder} with $m$ uprights is a diagram drawn in a rectangle, with
\begin{itemize}
\item $m$ parallel vertical lines running from the bottom edge to the top edge of the rectangle, oriented upwards,
\item some number of oriented horizontal {\bf rungs} connecting adjacent uprights,
\item a labelling of each interval (rungs or segments of uprights) by an integer between $0$ and $n$ inclusive,
\end{itemize}
such that the sum of labels (taken with signs according to the orientations of the intervals) at each trivalent vertex is zero.
\end{defn}

Now we introduce the category $\Lad_m^n $ of ladders. The objects are sequences of length $m$ of integers between $0$ and $n$ inclusive (that is, $n$-bounded weights of $\dU(\gl_m)$). The morphisms are linear combinations of ladders. The source of a ladder is the sequence of labels appearing on the lowest segments of the uprights, and the target is the sequence of labels appearing on the highest segments. Composition of morphisms is given by vertical concatenation of ladders.

Notice that, as $m$ varies, the categories $\Lad_m^n$ fit together as a tensor category $\Lad^n$, with tensor product given by horizontal juxtaposition.  In this tensor category the morphisms are generated by the single rung ladders.

Next, we define a functor from $\Lad_m^n \rightarrow \dU^n(\gl_m)$ which on objects is just the identity.  On morphisms we send rungs between the $i$-th and $(i+1)$-th uprights to divided powers as follows:
\begin{align*}
\begin{tikzpicture}[baseline=20]
\laddercoordinates{3}{1}
\node[below] at (l00) {$k_{i{-}1}$};
\node[below] at (l10) {$k_i$};
\node[below] at (l20) {$k_{i{+}1}$};
\node[below] at (l30) {$k_{i{+}2}$};
\node[above] at (l11) {$k_i{+}r$};
\node[above] at (l21) {$k_{i{+}1}{-}r$};
\node at ($(l00)+(-1,1)$) {$\cdots$};
\ladderI{0}{0};
\ladderFn{1}{0}{}{}{$r$}
\ladderI{3}{0};
\node at ($(l30)+(1,1)$) {$\cdots$};
\end{tikzpicture} & \mapsto E_i^{(r)} \one_{k_1\cdots k_m} \\
\intertext{and}
\begin{tikzpicture}[baseline=20]
\laddercoordinates{3}{1}
\node[below] at (l00) {$k_{i{-}1}$};
\node[below] at (l10) {$k_i$};
\node[below] at (l20) {$k_{i{+}1}$};
\node[below] at (l30) {$k_{i{+}2}$};
\node[above] at (l11) {$k_i{-}r$};
\node[above] at (l21) {$k_{i{+}1}{+}r$};
\node at ($(l00)+(-1,1)$) {$\cdots$};
\ladderI{0}{0};
\ladderEn{1}{0}{}{}{$r$}
\ladderI{3}{0};
\node at ($(l30)+(1,1)$) {$\cdots$};
\end{tikzpicture} & \mapsto F_i^{(r)} \one_{k_1\cdots k_m}
\end{align*}

%For example, figure \ref{fig:ladder-example} depicts the ladder which is mapped to $F_1^{(r)}F_2^{(t)}E_1^{(s)} \one_{k_1k_2k_3}$.

%\begin{figure}[ht]
%\begin{equation}\label{fig:ladder-example}
%\begin{ladder}{2}{3}
%\node[left] at (l00) {$k_1$};
%\node[left] at (l10) {$k_2$};
%\node[left] at (l20) {$k_3$};
%\ladderFn{0}{0}{$k_1{+}s$}{$k_2{-}s$}{$s$}
%\ladderEn{1}{1}{\small $k_2{-}s{-}t$}{$k_3{+}t$}{$t$}
%\ladderEn{0}{2}{$k_1{+}s{-}r$}{\small $k_2{-}s{-}t{+}r$}{$r$}
%\ladderI{0}{1}
%\ladderI{2}{0}
%\ladderI{2}{2}
%\end{ladder}
%\end{equation}
 %\end{figure}

\begin{prop} \label{prop:LaddertoU}
Under the functor above, $\dU^n(\gl_m)$ is the quotient of $\Lad_m^n$ by the following relations:
\begin{align}
\begin{tikzpicture}[baseline=40]
\laddercoordinates{2}{2}
\node[below] at (l00) {$k_1$};
\node[below] at (l10) {$k_2$};
\node[below] at (l20) {$k_3$};
\ladderEn{0}{0}{$k_1{-}r$}{$k_2{+}r$}{$r$}
\ladderFn{1}{1}{$k_2{+}r{+}s$}{$k_3{-}s$}{$s$}
\ladderIn{0}{1}{1}
\ladderIn{2}{0}{1}
\end{tikzpicture}
& =
\begin{tikzpicture}[baseline=40]
\laddercoordinates{2}{2}
\node[below] at (l00) {$k_1$};
\node[below] at (l10) {$k_2$};
\node[below] at (l20) {$k_3$};
\ladderEn{0}{1}{$k_1{-}r$}{$k_2{+}r{+}s$}{$r$}
\ladderFn{1}{0}{$k_2{+}s$}{$k_3{-}s$}{$s$}
\ladderIn{0}{0}{1}
\ladderIn{2}{1}{1}
\end{tikzpicture}
\label{eq:IHlad}
\displaybreak[1] \\
\begin{tikzpicture}[baseline=40]
\laddercoordinates{2}{2}
\node[below] at (l00) {$k_1$};
\node[below] at (l10) {$k_2$};
\node[below] at (l20) {$k_3$};
\ladderFn{0}{0}{$k_1{+}r$}{$k_2{-}r$}{$r$}
\ladderEn{1}{1}{$k_2{-}r{-}s$}{$k_3{+}s$}{$s$}
\ladderIn{0}{1}{1}
\ladderIn{2}{0}{1}
\end{tikzpicture}
& =
\begin{tikzpicture}[baseline=40]
\laddercoordinates{2}{2}
\node[below] at (l00) {$k_1$};
\node[below] at (l10) {$k_2$};
\node[below] at (l20) {$k_3$};
\ladderFn{0}{1}{$k_1{+}r$}{$k_2{-}r{-}s$}{$r$}
\ladderEn{1}{0}{$k_2{-}s$}{$k_3{+}s$}{$s$}
\ladderIn{0}{0}{1}
\ladderIn{2}{1}{1}
\end{tikzpicture}
\label{eq:IHlad2}
\displaybreak[1] \\
\tikz[baseline=40]{
\laddercoordinates{1}{2}
\ladderEn{0}{0}{$k-s$}{$l+s$}{$s$}
\ladderEn{0}{1}{$k-s-r$}{$l+s+r$}{$r$}
\node[left] at (l00) {$k$};
\node[right] at (l10) {$l$};
}
&=
\qBinomial{r+s}{r}
\tikz[baseline=20]{
\laddercoordinates{1}{1}
\ladderEn{0}{0}{$k-s-r$}{$l+s+r$}{$r+s$}
\node[left] at (l00) {$k$};
\node[right] at (l10) {$l$};
}
\label{eq:EE}
\displaybreak[1]
\\
\begin{tikzpicture}[baseline=40]
\laddercoordinates{1}{2}
\node[left] at (l00) {$k$};
\node[right] at (l10) {$l$};
\ladderEn{0}{0}{$k{-}s$}{$l{+}s$}{$s$}
\ladderFn{0}{1}{$k{-}s{+}r$}{$l{+}s{-}r$}{$r$}
\end{tikzpicture}
&= \sum_t \qBinomial{k-l+r-s}{t}
\begin{tikzpicture}[baseline=40]
\laddercoordinates{1}{2}
\node[left] at (l00) {$k$};
\node[right] at (l10) {$l$};
\ladderFn{0}{0}{$k{+}r{-}t$}{$l{-}r{+}t$}{$r{-}t$}
\ladderEn{0}{1}{$k{-}s{+}r$}{$l{+}s{-}r$}{$s{-}t$}
\end{tikzpicture}
\label{eq:EF=FE}
\end{align}
\vspace{-6mm}
\begin{equation}\label{eq:serre1}
\renewcommand{\ladderY}{1}
\begin{ladder}{2}{3}
\ladderE{0}{0}{}{}
\ladderE{0}{1}{}{}
\ladderE{1}{2}{}{}
\ladderI{0}{2}
\ladderIn{2}{0}{2}
\node[below] at (l00) {$k_1$};
\node[below] at (l10) {$k_2$};
\node[below] at (l20) {$k_3$};
\end{ladder}
- \qi{2}
\begin{ladder}{2}{3}
\ladderE{0}{0}{}{}
\ladderE{1}{1}{}{}
\ladderE{0}{2}{}{}
\ladderI{0}{1}
\ladderI{2}{0}
\ladderI{2}{2}
\node[below] at (l00) {$k_1$};
\node[below] at (l10) {$k_2$};
\node[below] at (l20) {$k_3$};
\end{ladder}
+
\begin{ladder}{2}{3}
\ladderE{1}{0}{}{}
\ladderE{0}{1}{}{}
\ladderE{0}{2}{}{}
\ladderI{0}{0}
\ladderIn{2}{1}{2}
\node[below] at (l00) {$k_1$};
\node[below] at (l10) {$k_2$};
\node[below] at (l20) {$k_3$};
\end{ladder}
= 0
\end{equation}
\end{prop}

\noindent together with the mirror reflection of (\ref{eq:EE}) and (\ref{eq:serre1}) (note that nothing happens to the coefficients of these equations under these operations). These relations are to be understood as containing arbitrarily many vertical strands on either side. Moreover, the horizontal rungs in (\ref{eq:serre1}) are all labelled $1$.

\begin{proof}
Since all $ E_i, F_i $ are in the image of the functor, we see that $ \Lad_m^n \rightarrow \dU^n(\gl_m) $ is full (it is obviously dominant).  It remains to see that the above relations generate the kernel.  To see this we need to check equations (\ref{rel:1},\ref{rel:2},\ref{rel:3},\ref{rel:4},\ref{rel:5}), along with the relation that $ \one_{\ul{k}} = 0 $ if $ \ul{k} $ is not an $ n$-bounded weight (although this last thing is clear from the definition of $\Lad_m^n$).

Equation (\ref{rel:1}) becomes (\ref{eq:EF=FE}) in diagrammatic form.  When $ |i-j| > 1 $, equation (\ref{rel:2}) is reflected in the isotopy invariance of ladders, while when $ |i-j|= 1$, (\ref{rel:2}) is (\ref{eq:IHlad}) and (\ref{eq:IHlad2}) in diagrammatic form. Relation (\ref{rel:3}) corresponds to (\ref{eq:serre1}) while (\ref{rel:4}) is again isotopy invariance. Finally, (\ref{rel:5}) corresponds to (\ref{eq:EE}).
\end{proof}

\subsection{Ladders as webs}\label{sec:psi}
There is a functor from $ \Lad_m^n \rightarrow \FSp(\SL_n)$  by forgetting the ladder structure of a ladder and thinking of it as a web. However, there is a slight discrepancy at the level of objects. More precisely, in $\Lad_m^n$ the objects $\ul{k}$ are sequences in $\{0,\ldots,n\}$, while in $\FSp(\SL_n)$ the objects are sequences in $\{1^+,\ldots,(n-1)^+\}$. The functor deletes $0$s and $n$s from the sequences, and sends $k$ to $k^+$.

\begin{prop}
\label{prop:psi}
The composition $\Lad_m^n \to \FSp(\SL_n) \to \Sp(\SL_n)$ can be factored through the functor $\Lad_m^n \to \dU^n(\gl_m)$ from the previous section, giving rise to a functor $\Psi_m^n : \dU^n(\gl_m) \to \Sp(\SL_n)$.
\end{prop}
\begin{proof}
To see that $ \Psi_m^n $ exists, we need only show that the diagrammatic relations of $ \dU^n(\gl_m) $ from Proposition \ref{prop:LaddertoU} are taken to the kernel of the functor $ \FSp(\SL_n) \to \Sp(\SL_n)$.

Relations (\ref{eq:EE}) and (\ref{eq:EF=FE}) hold in $\Sp(\SL_n) $ since they are exactly (\ref{eq:id1b}) and (\ref{eq:commutation}).

To show that (\ref{eq:IHlad}) holds in $ \Sp(\SL_n) $, we first observe that
\begin{equation*}
\begin{tikzpicture}[baseline]
\foreach \x/\y in {0/0,1/0,2/0,0/1,1/1,0/2} {
	\coordinate(z\x\y) at (\x+\y/2,\y/1.5);
}
\coordinate (z03) at (1,2);
\draw[mid>] (z00) node[below] {$1$} --  (z01);
\draw[mid>] (z01) -- node[left] {$k{+}1$} (z02);
\draw[mid>] (z10) node[below] {$k$} -- (z01);
\draw[mid>] (z20) node[below] {$1$} -- (z02);
\draw[mid>](z02) -- node[left] {$k{+}2$} (z03);
\end{tikzpicture}
 =
\begin{tikzpicture}[baseline]
\foreach \x/\y in {0/0,1/0,2/0,0/1,1/1,0/2} {
	\coordinate(z\x\y) at (\x+\y/2,\y/1.5);
}
\coordinate (z03) at (1,2);
\draw[mid>] (z00) node[below] {$1$} --  (z02);
\draw[mid>] (z10) node[below] {$k$} -- (z11);
\draw[mid>] (z20) node[below] {$1$} -- (z11);
\draw[mid>] (z11) -- node[right] {$k{+}1$} (z02);
\draw[mid>](z02) -- node[left] {$k{+}2$} (z03);
\end{tikzpicture}
\end{equation*}
 is a special case of (\ref{eq:IH}).
 This establishes (\ref{eq:IHlad}) in the special case $r=s=1$. For all other cases, we first use the previously established (\ref{eq:EE}) to replace the $r$ and $s$ rungs each with a collection of parallel $1$ rungs, and then repeatedly apply the special case.
 (Similarly for \eqref{eq:IHlad2} using the arrow reversal of \eqref{eq:IH}.)

Finally, relation (\ref{eq:serre1}) is exactly (\ref{eq:serre}) from Lemma \ref{lem:consequences}.
\end{proof}

We have now reached the situation described in the proof of the main result.  We have the diagram
\begin{equation}\label{diag:main2}
\xymatrix{
\Lad_m^n \ar[r] \ar[d] & \dU^n(\gl_m) \ar[dr]^{\Phi_m^n} \ar[d]_{\Psi_m^n} & \\
\FSp(\SL_n) \ar[r] & \Sp(\SL_n) \ar[r]^{\Gamma_n} & \Rep(\SL_n) \\
}
\end{equation}
Note that the left square of this diagram commutes by definition of $\Psi^n_m $.

\begin{prop}
\label{prop:commutes}
The right triangle of (\ref{diag:main2}) commutes.
\end{prop}

\begin{proof}
We proceed by an explicit calculation.  Consider $ E_j \one_{\ul{k}}$.  Via the map (\ref{eq:phimap}), we see that
$$
\Phi_m^n(E_j \one_{\ul{k}}): \Alt{k_1}{q} \bC_q^n \otimes \cdots \otimes \Alt{k_m}{q} \bC_q^n \rightarrow \Alt{l_1}{q} \bC_q^n \otimes \cdots \otimes \Alt{l_m}{q} \bC_q^n
$$
where $ \ul{l} = \ul{k} + \alpha_j$.

From Lemma \ref{lem:Eaction}, we see that
$$ \Phi_m^n(E_j \one_{\ul{k}}) = I^{\otimes j-1} \otimes M_{k_j,1} \otimes I^{\otimes m-j} \circ I^{\otimes j} \otimes M'_{1, k_{j+1} - 1} \otimes I^{\otimes m-j - 1} $$

On the other hand, consider the web
$$
w = \Psi^n_m(E_j \one_{\ul{k}}) =
\begin{tikzpicture}[baseline=20]
\laddercoordinates{3}{1}
\node[below] at (l10) {$k_j$};
\node[below] at (l20) {$k_{j+1}$};
\node[above] at (l11) {$k_j{+}1$};
\node[above] at (l21) {$k_{j+1}{-}1$};
\node at ($(l00)+(-1,1)$) {$\cdots$};
\ladderI{0}{0};
\ladderFn{1}{0}{}{}{$1$}
\ladderI{3}{0};
\node at ($(l30)+(1,1)$) {$\cdots$};
\end{tikzpicture}
$$

From the definition of $\Gamma_n $ in section \ref{sec:deffunctor} we see that
$$ \Gamma_n(w) = I^{\otimes j-1} \otimes M_{k_j,1} \otimes I^{\otimes m-j} \circ I^{\otimes j} \otimes M'_{1, k_{j+1} - 1} \otimes I^{\otimes m-j - 1}.$$
Thus $ \Gamma_n(\Psi_m^n(E_j \one_{\ul{k}})) = \Phi_m^n(E_j \one_{\ul{k}})$. A similar argument also holds for $F_j $ and since these generate $\dU(\gl_m)$ the result follows.
\end{proof}

\subsection{Surjectivity}

We would like to show that any web can be written, possibly using some relations, into ladder form.
This is not quite the case, simply because the boundary points of ladders are always oriented upwards. However, this is not a problem, because every object in $\Sp(\SL_n)$ is isomorphic, via a map made out of tags, to an upwards oriented one.

\begin{thm}
\label{thm:laddering}
Let $ D $ be a morphism in $ \Sp(\SL_n)$ between upwards oriented objects.  Then there exists $m \in {\mathbb N}$ and a morphism $ E \in \Lad_m^n $ such that $\Psi_m^n(E) = D $.
\end{thm}
\begin{proof}
Given any such diagrammatic morphism $D \in \Sp(\SL_n)$,
we first write express every tag in the diagram as the end of a strand labelled by $n$ connecting that tag to the edge of the diagram; anywhere this strand crosses an existing strand we interpret the crossing as a pair of trivalent vertices via Equation \eqref{eq:cancel-tags}. Then, if the total number of incoming and outgoing strands are not equal, we introduce new strands labelled by $0$ as needed to balance.

Now, just by a planar isotopy, we can write $D$ as
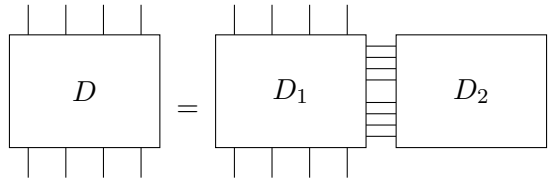
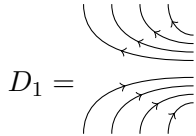
\begin{align*}
\begin{tikzpicture}[baseline=12]
\coordinate (a) at (0,0);
\coordinate (b) at (2,1.5);
\foreach \n/\x in {1/0.25,2/0.75, 3/1.25, 4/1.75} {
 \draw (\x,-0.4) coordinate (b\n) -- (\x, 1.9) coordinate (t\n);
}
\draw[fill=white] (a) rectangle (b);
\node at ($(a)!.5!(b)$) {$D$};
\end{tikzpicture}
&\;=\;
\begin{tikzpicture}[baseline=12]
\foreach \n/\x in {1/0.25,2/0.75, 3/1.25, 4/1.75} {
 \draw (\x,-0.4) coordinate (b\n) -- (\x, 1.9) coordinate (t\n);
}
\draw[fill=white] (a) rectangle node {$D_1$} (b);
\foreach \y in {0.6, 0.45, 0.3, 0.15} {
 \draw  (2,\y) -- (2.4, \y);
 \draw  (2,1.5-\y) -- (2.4, 1.5-\y);
}
\draw[fill=white] (2.4,0) rectangle node {$D_2$} (4.4,1.5);
\end{tikzpicture}
\end{align*}
where $$
D_1  =
\begin{tikzpicture}[baseline=7.5,x=0.75cm,y=0.75cm]
\foreach \n/\x/\y in {1/0.25/0.6,2/0.75/0.45, 3/1.25/0.3, 4/1.75/0.15} {
 \coordinate (b\n)  at  (\x, -0.4);
 \draw[mid>] (b\n) to[out=90,in=180] (2.2,\y);
 \coordinate (t\n) at (\x, 1.9) ;
 \draw[mid<] (t\n) to[out=-90,in=180] (2.2,1.5-\y);
}
\end{tikzpicture}
$$ and $D_2$ is in Morse position relative to the $x$-coordinate (that is, no two critical $x$-values or $x$-coordinates of vertices coincide) and further each trivalent vertex has two edges pointing to the left and one to the right. This can always be achieved at the expense of extra critical $x$-values in the strings.

Now replace $D_1$ with
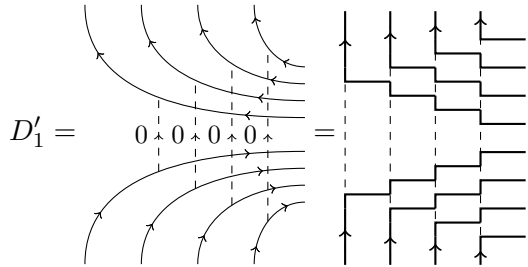
\begin{equation*}
D'_1 =
\begin{tikzpicture}[baseline=30,x=1.5cm, y=1.5cm]
\foreach \n/\x/\y in {1/0.25/0.6,2/0.75/0.45, 3/1.25/0.3, 4/1.75/0.15} {
 \coordinate (b\n)  at  (\x, -0.4);
 \draw[lower>, upper>] (b\n) to[out=90,in=180] coordinate (mb\n) (2.2,\y);
 \coordinate (t\n) at (\x, 1.9) ;
 \draw[lower<, upper<] (t\n) to[out=-90,in=180] coordinate (mt\n) (2.2,1.5-\y);
 \draw[dashed,mid>] (mb\n) --node[left=0.2pt] {\small $0$} (mt\n);
}
\end{tikzpicture}
=
\begin{tikzpicture}[baseline=10,x=0.6cm,y=0.375cm]
\foreach \x in {2,3,4,5} {
	\draw[dashed] (-\x,-3) -- (-\x,5);
}
\foreach \n in {1,2,3,4} {
	\foreach \m in {1,...,\n} {
		\draw[thick](-\m,5-\n+\m/2) -- +(-1,0) -- +(-1,0.5) node[coordinate] (z\n\m) {};
		\draw[thick](-\m,-3+\n-\m/2) -- +(-1,0) -- +(-1,-0.5) node[coordinate] (w\n\m) {};		
	}
	\draw[thick, mid>] (z\n\n) -- (-\n-1,5.5);
	\draw[thick, mid<] (w\n\n) -- (-\n-1,-3.5);
}
\end{tikzpicture}
\end{equation*}

Next, to the right of each elementary piece of the Morse decomposition of $D_2$, superimpose either a vertical $0$ strand or a vertical $n$ strand from the bottom to the top of the diagram and then make one of the following local replacements in the neighbourhood of that elementary piece:
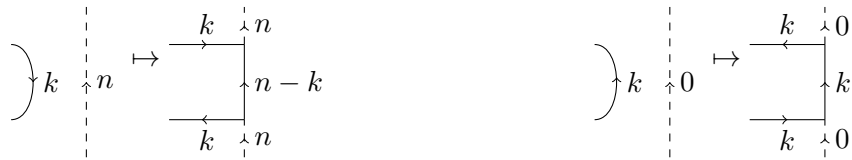
\begin{align}
\tikz[baseline=6]{
\draw[mid>] (-1,0.5) to[out=0,in=0] node[right] {$k$} (-1,-0.5);
\draw[mid>, dashed] (0,-1) --node[right] {$n$} (0,1);
} & \mapsto
\tikz[baseline=6]{
\draw[mid>, dashed] (0,-1) --node[right] {$n$} (0,-0.5);
\draw[mid>] (0,-0.5) --node[right] {$n-k$} (0,0.5);
\draw[mid>, dashed] (0,0.5) --node[right] {$n$} (0,1);
\draw[mid>] (-1,0.5) --node[above] {$k$} (0,0.5);
\draw[mid<] (-1,-0.5) --node[below] {$k$} (0,-0.5);
}
&
\tikz[baseline=6]{
\draw[mid<] (-1,0.5) to[out=0,in=0] node[right] {$k$} (-1,-0.5);
\draw[mid>, dashed] (0,-1) --node[right] {$0$} (0,1);
} & \mapsto
\tikz[baseline=6]{
\draw[mid>, dashed] (0,-1) --node[right] {$0$} (0,-0.5);
\draw[mid>] (0,-0.5) --node[right] {$k$} (0,0.5);
\draw[mid>, dashed] (0,0.5) --node[right] {$0$} (0,1);
\draw[mid<] (-1,0.5) --node[above] {$k$} (0,0.5);
\draw[mid>] (-1,-0.5) --node[below] {$k$} (0,-0.5);
}
\notag
\displaybreak[1]\\\notag\\
\tikz[baseline=6]{
\draw[mid>] (2,0.5) -- (1,0.5) to[out=180,in=180] node[left] {$k$} (1,-0.5) -- (2,-0.5);
\draw[mid>, dashed] (1.25,-1) --node[right] {$n$} (1.25,1);
} & \mapsto
\tikz[baseline=6]{
\draw[mid>, dashed] (0,-1) --node[left] {$n$} (0,-0.5);
\draw[mid>] (0,-0.5) --node[left] {$n-k$} (0,0.5);
\draw[mid>, dashed] (0,0.5) --node[left] {$n$} (0,1);
\draw[mid>] (1,0.5) --node[above] {$k$} (0,0.5);
\draw[mid<] (1,-0.5) --node[below] {$k$} (0,-0.5);
}
&
\tikz[baseline=6]{
\draw[mid<] (2,0.5) -- (1,0.5) to[out=180,in=180] node[left] {$k$} (1,-0.5) -- (2,-0.5);
\draw[mid>, dashed] (1.25,-1) --node[right] {$0$} (1.25,1);
} & \mapsto
\tikz[baseline=6]{
\draw[mid>, dashed] (0,-1) --node[left] {$0$} (0,-0.5);
\draw[mid>] (0,-0.5) --node[left] {$k$} (0,0.5);
\draw[mid>, dashed] (0,0.5) --node[left] {$0$} (0,1);
\draw[mid<] (1,0.5) --node[above] {$k$} (0,0.5);
\draw[mid>] (1,-0.5) --node[below] {$k$} (0,-0.5);
}
\notag
\displaybreak[1]\\\notag\\
\tikz[baseline=6]{
\draw[mid>] (-1,0.5) node[left] {$k$} -- (-0.5,0);
\draw[mid>] (-1,-0.5) node[left] {$l$}-- (-0.5,0);
\draw[upper>] (-0.5,0) -- (1,0) node[right] {$k{+}l$};
\draw[dashed,lower>] (0,-1) node[below] {0} -- (0,1) node[above] {0};
}
& \mapsto
\tikz[baseline=6]{
\draw[mid>] (-1,-0.5) node[left] {$l$} -- (0,-0.5);
\draw[mid>] (-1,0) node[left] {$k$} -- (0,0);
\draw[mid>] (0,0.5) -- (1,0.5) node[right] {$k{+}l$} ;
\draw[dashed,mid>] (0,-1) node[below] {0} -- (0,-0.5);
\draw[mid>] (0,-0.5) -- (0,0);
\draw[mid>] (0,0) -- (0,0.5);
\draw[dashed,mid>] (0,0.5) -- (0,1) node[above] {0};
}
&
\tikz[baseline=6]{
\draw[mid<] (-1,0.5) node[left] {$k$} -- (-0.5,0);
\draw[mid<] (-1,-0.5) node[left] {$l$}-- (-0.5,0);
\draw[upper<] (-0.5,0) -- (1,0) node[right] {$k{+}l$};
\draw[dashed,lower>] (0,-1) node[below] {0} -- (0,1) node[above] {0};
}
& \mapsto
\tikz[baseline=6]{
\draw[mid<] (-1,0) node[left] {$l$} -- (0,0);
\draw[mid<] (-1,0.5) node[left] {$k$} -- (0,0.5);
\draw[mid<] (0,-0.5) -- (1,-0.5) node[right] {$k{+}l$} ;
\draw[dashed,mid>] (0,-1) node[below] {0} -- (0,-0.5);
\draw[mid>] (0,-0.5) -- (0,0);
\draw[mid>] (0,0) -- (0,0.5);
\draw[dashed,mid>] (0,0.5) -- (0,1) node[above] {0};
}
\notag
\displaybreak[1]\\\notag\\
\tikz[baseline=6]{
\draw[mid<] (-1,0.5) node[left] {$k{+}l$} -- (-0.5,0);
\draw[mid>] (-1,-0.5) node[left] {$k$}-- (-0.5,0);
\draw[upper<] (-0.5,0) -- (1,0) node[right] {$l$};
\draw[dashed,lower>] (0,-1) node[below] {0} -- (0,1) node[above] {0};
}
& \mapsto
\tikz[baseline=6]{
\draw[mid>] (-1,-0.5) node[left] {$k$} -- (0,-0.5);
\draw[mid<] (-1,0.5) node[left] {$k{+}l$} -- (0,0.5);
\draw[mid<] (0,0) -- (1,0) node[right] {$l$} ;
\draw[dashed,mid>] (0,-1) node[below] {0} -- (0,-0.5);
\draw[mid>] (0,-0.5) -- (0,0);
\draw[mid>] (0,0) -- (0,0.5);
\draw[dashed,mid>] (0,0.5) -- (0,1) node[above] {0};
}
&
\tikz[baseline=6]{
\draw[mid<] (-1,0.5) node[left] {$k$} -- (-0.5,0);
\draw[mid>] (-1,-0.5) node[left] {$k{+}l$}-- (-0.5,0);
\draw[upper>] (-0.5,0) -- (1,0) node[right] {$l$};
\draw[dashed,lower>] (0,-1) node[below] {0} -- (0,1) node[above] {0};
}
& \mapsto
\tikz[baseline=6]{
\draw[mid>] (-1,-0.5) node[left] {$k{+}l$} -- (0,-0.5);
\draw[mid<] (-1,0.5) node[left] {$k$} -- (0,0.5);
\draw[mid>] (0,0) -- (1,0) node[right] {$l$} ;
\draw[dashed,mid>] (0,-1) node[below] {0} -- (0,-0.5);
\draw[mid>] (0,-0.5) -- (0,0);
\draw[mid>] (0,0) -- (0,0.5);
\draw[dashed,mid>] (0,0.5) -- (0,1) node[above] {0};
}
\notag
\displaybreak[1]\\\notag\\
\tikz[baseline=6]{
\draw[mid>] (-1,0.5) node[left] {$k$} -- (-0.5,0);
\draw[mid<] (-1,-0.5) node[left] {$k{+}l$}-- (-0.5,0);
\draw[upper<] (-0.5,0) -- (1,0) node[right] {$l$};
\draw[dashed,lower>] (0,-1) node[below] {$n$} -- (0,1) node[above] {$n$};
}
& \mapsto
\tikz[baseline=6]{
\draw[mid<] (-1,-0.5) node[left] {$k{+}l$} -- (0,-0.5);
\draw[mid>] (-1,0.5) node[left] {$k$} -- (0,0.5);
\draw[mid<] (0,0) -- (1,0) node[right] {$l$} ;
\draw[dashed,mid>] (0,-1) node[below] {$n$} -- (0,-0.5);
\draw[mid>] (0,-0.5) -- (0,0);
\draw[mid>] (0,0) -- (0,0.5);
\draw[dashed,mid>] (0,0.5) -- (0,1) node[above] {$n$};
}
&
\tikz[baseline=6]{
\draw[mid>] (-1,0.5) node[left] {$k{+}l$} -- (-0.5,0);
\draw[mid<] (-1,-0.5) node[left] {$k$}-- (-0.5,0);
\draw[upper>] (-0.5,0) -- (1,0) node[right] {$l$};
\draw[dashed,lower>] (0,-1) node[below] {$n$} -- (0,1) node[above] {$n$};
}
& \mapsto
\tikz[baseline=6]{
\draw[mid<] (-1,-0.5) node[left] {$k$} -- (0,-0.5);
\draw[mid>] (-1,0.5) node[left] {$k{+}l$} -- (0,0.5);
\draw[mid>] (0,0) -- (1,0) node[right] {$l$} ;
\draw[dashed,mid>] (0,-1) node[below] {$n$} -- (0,-0.5);
\draw[mid>] (0,-0.5) -- (0,0);
\draw[mid>] (0,0) -- (0,0.5);
\draw[dashed,mid>] (0,0.5) -- (0,1) node[above] {$n$};
}
\label{eq:n-right-of-trivalent}
\end{align}
and then finally replacing each other instance where a superimposed strand crosses a horizontal strand of $D_2$ as follows
\begin{align*}
\tikz[baseline=6]{
\draw[mid>] (-1,0) -- (0,0);
\draw (0,0) --node[below] {$k$} (1,0);
\draw[mid>, dashed] (0,-1) -- (0,0);
\draw[dashed] (0,0) --node[right] {0} (0,1);
} & \mapsto
\tikz[baseline=6]{
\draw[mid>] (-1,-0.2) -- (0,-0.2);
\draw[mid>] (0,-0.2) -- (0,0.2);
\draw[mid>] (0,0.2) --node[below] {$k$} (1,0.2);
\draw[mid>,dashed]  (0,-1) -- (0,-0.2);
\draw[dashed] (0,0.2) --node[right] {0} (0,1);
}
&
\tikz[baseline=6]{
\draw[mid<] (-1,0) -- (0,0);
\draw (0,0) --node[below] {$k$} (1,0);
\draw[mid>, dashed] (0,-1) -- (0,0);
\draw[dashed] (0,0) --node[right] {0} (0,1);
} & \mapsto
\tikz[baseline=6]{
\draw[mid<] (-1,0.2) -- (0,0.2);
\draw[mid>] (0,-0.2) -- (0,0.2);
\draw[mid<] (0,-0.2) --node[below] {$k$} (1,-0.2);
\draw[mid>,dashed]  (0,-1) -- (0,-0.2);
\draw[dashed] (0,0.2) --node[right] {0} (0,1);
} \displaybreak[1] \\\notag\\
\tikz[baseline=6]{
\draw[mid>] (-1,0) -- (0,0);
\draw (0,0) --node[below] {$k$} (1,0);
\draw[mid>, dashed] (0,-1) -- (0,0);
\draw[dashed] (0,0) --node[right] {$n$} (0,1);
} & \mapsto
\tikz[baseline=6]{
\draw[mid>] (-1,0.3) -- node[above] {$k$} (0,0.3);
\draw[mid>] (0,-0.3) --node[left] {$n{-}k$} (0,0.3);
\draw[mid>] (0,-0.3) --node[below] {$k$} (1,-0.3);
\draw[mid>,dashed]  (0,-1) -- (0,-0.3);
\draw[mid>,dashed] (0,0.3) --node[right] {$n$} (0,1);
}
&
\tikz[baseline=6]{
\draw[mid<] (-1,0) -- (0,0);
\draw (0,0) --node[below] {$k$} (1,0);
\draw[mid>, dashed] (0,-1) -- (0,0);
\draw[dashed] (0,0) --node[right] {$n$} (0,1);
} & \mapsto
\tikz[baseline=6]{
\draw[mid<] (-1,-0.3) --node[below] {$k$} (0,-0.3);
\draw[mid>] (0,-0.3) --node[left] {$n{-}k$} (0,0.3);
\draw[mid<] (0,0.3) --node[below] {$k$} (1,0.3);
\draw[mid>,dashed]  (0,-1) -- (0,-0.3);
\draw[mid>,dashed] (0,0.3) --node[right] {$n$} (0,1);
}
\end{align*}
to obtain $D'_2$. Now one can check that $D'_2$ is in fact equal to $D_2$, using only a few relations from the spider. In particular, for each of the replacements above involving a $0$ strand, when we delete the $0$ strands we see that nothing has changed. In the replacements involving an $n$ strand but no trivalent vertices, after removing the $n$-strands and replacing their endpoints with tags, we find we can cancel the tags according to Equation \eqref{eq:cancel-tags}. Finally, in the replacements involving an $n$ strand to the right of a trivalent vertex (those in \eqref{eq:n-right-of-trivalent}), we need to use Equations \eqref{eq:tag-migration} and \eqref{eq:tag-migration2} to move one tag past the trivalent vertex, and then Equation \eqref{eq:cancel-tags} to cancel them.

In each of the local replacements used to form $D'_2$ the new diagram consists of part of an upright of the ladder, along with several `half-rungs'. It is easy to see that all of these half-rungs come in matching pairs forming complete rungs, except at the left margin of $D'_2$. Similarly, $D'_1$ is a ladder except that it has half-rungs along its right margin. The horizontal juxtaposition $D'_1 D'_2$ is then a ladder.
Since $D$ is equivalent in $\Sp(\SL_n)$ to $D'_1 D'_2$, we are finished.
\end{proof}

\section{Braidings}
Let $ \dU^n(\gl_\bullet) = \bigoplus_{m=1}^\infty \dU^n(\gl_m) $.  This is a category whose objects are $n$-bounded weights $\ul{k} = (k_1, \dots, k_m) $ and whose morphisms are given by
$$ \Hom((k_1, \dots, k_m), (l_1, \dots, l_p)) = \begin{cases} \one_{(l_1, \dots, l_p)}  \dU^n(\gl_m) \one_{(k_1, \dots, k_m)} \ \ & \text{ if } m =p \\
0 \ \ & \text{ otherwise.}
\end{cases}
$$
The functors $ \Phi^n_m $ combine together to a functor $ \Phi^n : \dU^n(\gl_\bullet) \rightarrow \Rep(\SL_n) $ which factors through $ \Sp(\SL_n)$. Our goal in this section is to define a braided monoidal category structure on $ \dU^n(\gl_\bullet) $ and to show that $\Phi^n $ preserves the braiding.

\subsection{Braided monoidal category structure on $\dU^n(\gl_\bullet)$} \label{se:braiding}
First, we define a monoidal category structure on $\dU^n(\gl_\bullet)$.  The tensor product of objects is given by concatenation $(k_1, \dots, k_m) \otimes (l_1, \dots, l_p) = (k_1, \dots, k_m, l_1, \dots, l_p) $.  The tensor product of morphisms comes from the obvious embedding $ U_q(\gl_{m}) \otimes U_q(\gl_{p}) \rightarrow U_q(\gl_{m+p})$.  From the perspective of ladders, the tensor product is given by horizontal juxtaposition.  This monoidal structure is associative with trivial associator.

Recall that a braiding $\beta$ on a monoidal category $ \mathcal C$ is system of natural isomorphisms $\beta_{V, W} : V \otimes W \rightarrow W \otimes V $ satisfying the ``hexagon equations''
\begin{align} \label{eq:hex}
\beta_{U \otimes V, W} & = \beta_{U, W}  \otimes I_V \circ I_U \otimes \beta_{V, W} \\
\intertext{and}
\beta_{U, V \otimes W} & = I_V \otimes \beta_{U, W} \circ \beta_{U, V}  \otimes I_W \notag
\end{align}
We will now define a braiding on $ \dU^n(\gl_\bullet) $.

For each $n$-bounded weight $ \ul{k} $ and for $ 1 \le i < m $, we set $ s_i(\ul{k}) = (k_1, \dots, k_{i+1}, k_i, \dots, k_m) $.  Following Lusztig \cite[5.2.1]{MR1227098}, we define $ T''_{i,-1} \one_{\ul{k}} \in \one_{s_i(\ul{k})} \dU^n(\gl_m) \one_{\ul{k}} $ by the formula
\begin{equation}\label{eq:temp3}
T''_{i,-1} \one_{\ul{k}} = \sum_{a, b, c \ge 0, -a +b -c = k_i - k_{i+1}} (-1)^b q^{ac - b} E_i^{(a)} F_i^{(b)} E_i^{(c)} \one_{\ul{k}}
\end{equation}
Note that these sums are actually finite since we are working in the truncation $ \dU^n(\gl_m) $.

%The summation in \ref{eq:temp3} above when $c \ne 0$ is actually equal to zero. This can be shown by a direct calculation -- namely by commuting the $F_i^{(b)}$ past the $E_i^{(a)}$ or $E_i^{(c)}$ and simplifying. As a result you end up with the simplified expression \todo{hanging sentence here?}

Lusztig's elements admit the following simplified form, a fact which seems to have been first observed by Chuang-Rouquier \cite{MR2373155} when $ q = 1 $.
\begin{lem} \label{le:singlesum}
$T''_{i,-1} \one_{\ul{k}} = \sum_{a, b \ge 0, -a +b = k_i - k_{i+1}} (-q)^{-b} E_i^{(a)} F_i^{(b)} \one_{\ul{k}} $
\end{lem}
\begin{proof}
By the definition of $ T''_{i,-1}$, we must show that
$$\sum_{a, b, c \ge 0, -a +b -c = k_i - k_{i+1}, c > 0 } (-1)^b q^{ac - b} E_i^{(a)} F_i^{(b)} E_i^{(c)} \one_{\ul{k}} = 0 $$
This is easily proven by direct calculation starting with the usual commutation relation $ (E_i F_i - F_i E_i)\one_{\ul{k}} = [k_i - k_{i+1}]_q $.
\end{proof}

We will need the following modification of Lusztig's definition in order to later match with the braiding on $ \Rep(\SL_n) $,
$$ T_i \one_{\ul{k}} := (-1)^{k_i+k_i k_{i+1}} q^{k_i - \frac{k_i k_{i+1}}{n}} T''_{i,-1} \one_{\ul{k}}. $$

\begin{lem} \label{le:braid}
The elements $ T_i $ are invertible.  Moreover they satisfy the braid relations
$$ T_{i+1} T_{i} T_{i+1} \one_{\ul{k}} = T_i T_{i+1} T_i \one_{\ul{k}} \quad \text{ and } \quad T_i T_j \one_{\ul{k}} = T_j T_i \one_{\ul{k}}, \text{ if } |i - j| \ge 2. $$
\end{lem}
\begin{proof}
Lusztig proved that the $T''_{i,-1} $ are invertible (section 5.2.3 in \cite{MR1227098}) and that they satisfy the braid relations (section 39.4.1 in \cite{MR1227098}).  The corresponding result for the $ T_i $ follows immediately.
\end{proof}

From the lemma, we can define $ T_w \one_{\ul{k}} $ for any $ w\in S_m $ by using the usual lift of $ S_m $ into the braid group $ B_m $.

Now we are in a position to define the braiding on our category.  For any two objects $ \ul{k} = (k_1, \dots, k_m), \ul{l} = (l_1, \dots, l_p) $, we define $ \beta_{\ul{k}, \ul{l}} = T_w \one_{(k_1, \dots, k_m, l_1, \dots, l_p)} $ where $ w \in S_{m+p} $ is defined by
$$ w(i) = \begin{cases} p + i \ \ & \text{ if } i \le m  \\
 i - m \ \ & \text{ if } i > m.
 \end{cases}
 $$

\begin{lem} \label{le:natural}
The map $\beta $ is a natural transformation from the bifunctor $- \bigotimes -$ to the bifunctor $- \bigotimes^{\text{op}} -$.
\end{lem}

\begin{proof}
We must prove that for any morphism $ \phi : \ul{k} \rightarrow \ul{k'} $ in $\dU^n(\gl_m) $ we have $ \beta_{\ul{k'}, \ul{l}} \circ (\phi \otimes I) = (I \otimes \phi) \circ \beta_{\ul{k}, \ul{l}} $.

Since the morphisms are generated by the $E_i,F_i $, it suffices to prove the result when $ \phi $ is an $ E_i $ or $ F_i $.  Because of the definition of $ T_w$, it suffices to prove that
$$ T_i T_j E_i = E_j T_i T_j \text{ when $|i-j| = 1$ and } T_j E_i = E_i T_j \text{ when $ |i-j| \ge 2$} $$
along with the same equations when $E_i $ is replaced by $ F_i $.  The second equation follows from the definition of $T_j$ and the commutativity of $ E_i, E_j $ when $ |i-j| \ge 2 $.

Thus it suffices to prove the first equation.  Note that this equation is equivalent to
$$
T_j E_i T_j^{-1} = T_i^{-1} E_j T_i
$$
The analogous equation with $ T_i $ replaced by $ T''_{i,-1} $ was proven by Lusztig, section 39.2.4 of \cite{MR1227098}.  The result for $ T_i $ follows immediately.
\end{proof}

\begin{thm}
This defines a braided monoidal category structure on $ \dU^n(\gl_\bullet)$.
\end{thm}

\begin{proof}
First, by Lemma \ref{le:natural}, $\beta $ is a natural transformation and by Lemma \ref{le:braid}, $\beta $ is a natural isomorphism.  The hexagon equations hold by the definition of $ \beta $ and thus we conclude that $ \beta $ gives a braided monoidal category structure.
\end{proof}

\subsection{Comparison of braidings}
Recall that $ \Rep(\SL_n) $ carries the structure of a braided monoidal category using the usual $R$-matrix.  We will use $ \beta $ to denote the braiding. Our goal now is to prove that
\begin{thm} \label{th:braiding}
The functor $ \Phi^n: \dU^n(\gl_\bullet) \rightarrow \Rep(\SL_n) $ is a functor of braided monoidal categories.
\end{thm}

\begin{proof}
Clearly, $\Phi^n$ is a tensor functor.  So we need to show that $ \Phi^n $ carries the braiding in $ \dU^n(\gl_\bullet) $ to the braiding in $\Rep(\SL_n) $.  By the hexagon equations (\ref{eq:hex}), it suffices to prove that if $ 0 \le k, l \le m $ then $$ \Phi^n(\beta_{k, l}) = \beta_{\alt^{k}_{q} \mathbb C_q^n, \alt^{l}_{q} \mathbb C_q^n}. $$

First we claim that when $ k_i = 0 $ or $ k_{i+1} = 0 $, then $ \Phi^n(T_i \one_{\ul{k}}) $ is the identity.  To see this recall that Lusztig proved (Proposition 5.2.2(b) in \cite{MR1227098}) that if $ F_i v = 0 $ and $ H_i v = -r v $, then $ T''_{i,-1}(v) = E^{(r)} v $.  On the other hand, Lusztig proved (in the same Proposition together with Proposition 5.2.3(b)), that if $ E_i v = 0 $ and $H_i v = rv $, then $ T''_{i,-1}(v) = (-1)^r q^{-r} F_i^{(r)} v $.  From this it follows that $ T_i \one_{\ul{k}} = E_i^{(k_{i+1})} \one_{\ul{k}} $ if $ k_i = 0 $ and $ T_i \one_{\ul{k}} = F_i^{(k_i)} $ if $ k_{i+1} = 0 $.  Thus the claim follows.

From the above claim, we see that if we consider the weights $ (k, 0^{k-1}) = (k, 0, \dots, 0)$ and similarly $ (l, 0^{l-1}) $ then we have $ \Phi^n(\beta_{(k, 0^{k-1}), (l, 0^{l-1})}) =  \Phi^n(\beta_{k, l})$.

Now consider weights $ 1^k = (1, \dots, 1) $ where there are $ k $ $1$s and similarly $ 1^l$.  We define morphisms $ \phi: (k, 0^{k-1})  \rightarrow 1^k$ and  $ \psi : (l, 0^{l-1}) \rightarrow 1^{l} $ by
$$
\phi = F_1 F_2 F_1 \cdots F_{k-1} \cdots F_1 \one_{(k, 0^{k-1})}
$$
and similarly for $ \psi $.

Note that $ \Phi^n(\phi) $ is the injective map $ \Alt{k}{q} \mathbb C_q^n \rightarrow (\mathbb C_q^n)^{\otimes k} $.  By the naturality of $ \beta $, we have $  \beta_{1^k, 1^l} \circ \phi \otimes \psi =  \psi \otimes \phi \circ \beta_{(k,0^{k-1}), (l, 0^{l-1})} $.  Thus we have
$$
  \Phi^n(\beta_{1^k, 1^l})\circ \Phi^n(\phi) \otimes \Phi^n(\psi) = \Phi^n(\psi) \otimes \Phi^n(\phi) \circ \Phi^n(\beta_{k,l})
$$

From Lemma \ref{le:standardbraid} below, we see that $ \Phi^n(\beta_{1,1}) = \beta_{\mathbb C_q^n, \mathbb C_q^n} $ and thus by the hexagon equations (\ref{eq:hex}), $\Phi^n(\beta_{1^k, 1^l}) = \beta_{(\mathbb C_q^n)^{\otimes k}, (\mathbb C_q^n)^{\otimes l}} $.  So the above equation becomes
$$
 \beta_{(\mathbb C_q^n)^{\otimes k}, (\mathbb C_q^n)^{\otimes l}} \circ \Phi^n(\phi) \otimes \Phi^n(\psi) = \Phi^n(\psi) \otimes \Phi^n(\phi) \circ  \Phi^n(\beta_{k,l})
$$
On the other hand, by the naturality of the braiding in $\Rep(\SL_n)$, we have
$$
 \beta_{(\mathbb C_q^n)^{\otimes k}, (\mathbb C_q^n)^{\otimes l}} \circ \Phi^n(\phi) \otimes \Phi^n(\psi) = \Phi^n(\psi) \otimes \Phi^n(\phi) \circ \beta_{\alt^{k}_{q} \mathbb C_q^n, \alt^{l}_{q} \mathbb C_q^n}.
$$
Thus the injectivity of $ \Phi^n(\psi) \otimes \Phi^n(\phi)$ implies the desired result.
\end{proof}

\begin{lem} \label{le:standardbraid}
$\Phi^n(\beta_{1,1}) = \beta_{\mathbb C_q^n, \mathbb C_q^n}$

\end{lem}

\begin{proof}
This follows by a direct computation.  First, it is a standard fact that $  \beta_{\mathbb C_q^n, \mathbb C_q^n} $ acts by $ q^{1-1/n} $ on $ S^2_q (\mathbb C_q^n) $ and by $-q^{-1-1/n} $ on $ \Alt{2}{q} (\mathbb C_q^n)$. (To see this, we use the fact that the eigenvalue of the square of the braiding on the $Y$ summand of $X \otimes X$ is $\theta(Y)\theta(X)^{-2}$, where $\theta(V_\lambda) = q^{\langle\lambda,\lambda+2\rho\rangle}$ is the twist factor. See \cite[\S 1.1.4 and \S 3.5]{MR2783128}.) So it suffices to check that $\Phi^n(\beta_{1,1}) $ does the same thing.

Now we claim that $ T''_{1,-1} $ acts by 1 on $ S^2_q(\bC_q^n) $ and acts by $ -q^{-2} $ on $ \Alt{2}{q}(\bC_q^n) $. This is because we have the following action of $U_q(\gl_2)$
$$\xymatrix{
\Alt{2}{q}(\bC_q^n) \ar@/^1pc/[r]^E & \bC_q^n \otimes \bC_q^n \ar@/^1pc/[l]^F \ar@/^1pc/[r]^E & \Alt{2}{q}(\bC_q^n) \ar@/^1pc/[l]^F }
$$
where $T''_{1,-1} \one_{(1,1)} = \one_{(1,1)} - q^{-1}EF \one_{(1,1)}$. Thus $F$ acts by zero on the summand $S^2_q(\bC_q^n) \subset \bC_q^n \otimes \bC_q^n$ which means $T''_{1,-1}$ acts by $1$. Also,
$$T''_{1,-1}(E) = (E - q^{-1}EFE) = E - q^{-1}(q+q^{-1})E = -q^{-2}E$$
which means that $T''_{1,-1}$ acts by multiplication by $-q^{-2}$ on the summand $\Alt{2}{q}(\bC_q^n) \subset \bC_q^n \otimes \bC_q^n$.

From the above computation of $ T''_{1,-1} $ we see that $\Phi^n(\beta_{1,1})$ acts by $ q^{1-1/n} $ on $ S^2_q (\mathbb C_q^n) $ and by $-q^{-1-1/n} $ on $ \Alt{2}{q} (\mathbb C_q^n)$ as desired.
\end{proof}

Using the form for $ T''_{1,-1} $ given in Lemma \ref{le:singlesum}, we can translate Theorem \ref{th:braiding} to the language of webs as follows.
\begin{cor}
The braiding $ \beta_{\alt^{k}_{q} \mathbb C_q^n, \alt^{l}_{q} \mathbb C_q^n } $ is given by the following sum of webs

$$ (-1)^{k+kl} q^{k-\frac{kl}{n}} \sum_{\substack{a,b\geq 0 \\ b-a = k-l}} (-q)^{-b}
\begin{tikzpicture}[baseline=40]
\laddercoordinates{1}{2}
\node[left] at (l00) {$k$};
\node[right] at (l10) {$l$};
\ladderEn{0}{0}{$k{-}b$}{$l{+}b$}{$b$}
\ladderFn{0}{1}{$l$}{$k$}{$a$}
\end{tikzpicture} $$
\end{cor}

% ----------------------------------------------------------------
%\newcommand{\urlprefix}{}
\bibliographystyle{alpha}
\bibliography{bibliography/bibliography}

\newcommand{\noopsort}[1]{}\def\cprime{$'$} \def\cprime{$'$}
\begin{thebibliography}{MOY98}

\bibitem[BZ08]{BZ}
Arkady Berenstein and Sebastian Zwicknagl.
\newblock Braided symmetric and exterior algebras.
\newblock {\em Trans. Amer. Math. Soc.}, 360(7):3429--3472, 2008.
\newblock \arxiv{math/0504155} \mathscinet{MR2386232}
  \doi{10.1090/S0002-9947-08-04373-0}.

\bibitem[CKL10]{MR2593278}
Sabin Cautis, Joel Kamnitzer, and Anthony Licata.
\newblock Categorical geometric skew {H}owe duality.
\newblock {\em Invent. Math.}, 180(1):111--159, 2010.
\newblock \arxiv{0902.1795} \mathscinet{MR2593278}
  \doi{10.1007/s00222-009-0227-1}.

\bibitem[CR08]{MR2373155}
Joseph Chuang and Rapha{\"e}l Rouquier.
\newblock Derived equivalences for symmetric groups and {$\mathfrak
  {sl}_2$}-categorification.
\newblock {\em Ann. of Math. (2)}, 167(1):245--298, 2008.
\newblock \arxiv{math/0407205} \mathscinet{MR2373155}
  \doi{10.4007/annals.2008.167.245}.

\bibitem[Dot03]{MR1990659}
Stephen Doty.
\newblock Presenting generalized {$q$}-{S}chur algebras.
\newblock {\em Represent. Theory}, 7:196--213 (electronic), 2003.
\newblock \arxiv{math/0305208} \mathscinet{MR1990659}
  \doi{10.1090/S1088-4165-03-00176-6}.

\bibitem[Gra12]{1212.4511}
Jonathan Grant.
\newblock The moduli problem of {Lobb} and {Zentner} and the coloured sl(n)
  graph invariant, 2012.
\newblock \arxiv{1212.4511}.

\bibitem[How89]{MR986027}
Roger Howe.
\newblock Remarks on classical invariant theory.
\newblock {\em Trans. Amer. Math. Soc.}, 313(2):539--570, 1989.
\newblock \mathscinet{MR986027} \doi{10.2307/2001418}.

\bibitem[How95]{MR1321638}
Roger Howe.
\newblock Perspectives on invariant theory: {S}chur duality, multiplicity-free
  actions and beyond.
\newblock In {\em The {S}chur lectures (1992) ({T}el {A}viv)}, volume~8 of {\em
  Israel Math. Conf. Proc.}, pages 1--182. Bar-Ilan Univ., Ramat Gan, 1995.
\newblock \mathscinet{MR1321638}.

\bibitem[JK12]{math/0506403}
Myeong-Ju Jeong and Dongseok Kim.
\newblock The quantum $\mathfrak{sl}(n,\mathbb{C})$ representation theory and
  its applications.
\newblock {\em Journal of the Korean Mathematical Society}, 49(5):993--1015,
  2012.
\newblock \arxiv{math/0506403}
  \doi{http://dx.doi.org/10.4134/JKMS.2012.49.5.993}.

\bibitem[Kho04]{MR2100691}
Mikhail Khovanov.
\newblock sl(3) link homology.
\newblock {\em Algebr. Geom. Topol.}, 4:1045--1081, 2004.
\newblock \arxiv{math.QA/0304375} \mathscinet{MR2100691}
  \doi{10.2140/agt.2004.4.1045}.

\bibitem[Kim03]{math.QA/0310143}
Dongseok Kim.
\newblock Graphical calculus on representations of quantum lie algebras, 2003.
\newblock {P}h. {D}. thesis, University of California, Davis
  \arxiv{math.QA/0310143}.

\bibitem[KL10]{MR2628852}
Mikhail Khovanov and Aaron~D. Lauda.
\newblock A categorification of quantum {${\rm sl}(n)$}.
\newblock {\em Quantum Topol.}, 1(1):1--92, 2010.
\newblock \arxiv{0807.3250} \mathscinet{MR2628852} \doi{10.4171/QT/1}.

\bibitem[Kup96]{MR1403861}
Greg Kuperberg.
\newblock Spiders for rank {$2$} {L}ie algebras.
\newblock {\em Comm. Math. Phys.}, 180(1):109--151, 1996.
\newblock \arxiv{q-alg/9712003} \mathscinet{MR1403861}
  \euclid{euclid.cmp/1104287237}.

\bibitem[LRQ]{1212.6076}
Aaron Lauda, David Rose, and Hoel Queffelec.
\newblock {K}hovanov homology is a skew {H}owe $2$-representation of
  categorified quantum $\mathfrak{sl}_m$.
\newblock \arxiv{1212.6076}.

\bibitem[Lus93]{MR1227098}
George Lusztig.
\newblock {\em Introduction to quantum groups}, volume 110 of {\em Progress in
  Mathematics}.
\newblock Birkh\"auser Boston Inc., Boston, MA, 1993.
\newblock \mathscinet{MR1227098} \doi{10.1007/978-0-8176-4717-9}.

\bibitem[Mor07]{0704.1503}
Scott Morrison.
\newblock {\em A Diagrammatic Category for the Representation Theory of
  $U_q\left(\mathfrak{sl}_n\right)$}.
\newblock PhD thesis, University of California, Berkeley, 2007.
\newblock \arxiv{0704.1503}.

\bibitem[MOY98]{MR1659228}
Hitoshi Murakami, Tomotada Ohtsuki, and Shuji Yamada.
\newblock Homfly polynomial via an invariant of colored plane graphs.
\newblock {\em Enseign. Math. (2)}, 44(3-4):325--360, 1998.
\newblock \mathscinet{MR1659228}.

\bibitem[MPS11]{MR2783128}
Scott Morrison, Emily Peters, and Noah Snyder.
\newblock Knot polynomial identities and quantum group coincidences.
\newblock {\em Quantum Topol.}, 2(2):101--156, 2011.
\newblock \doi{10.4171/QT/16} \mathscinet{MR2783128} \arxiv{1003.0022}.

\bibitem[MPT12]{1206.2118}
Marco Mackaay, Weiwei Pan, and Daniel Tubbenhauer.
\newblock The $\mathfrak{sl}_3$ web algebra, 2012.
\newblock \arxiv{1206.2118}.

\bibitem[MSV09]{MR2491657}
Marco Mackaay, Marko Sto{\v{s}}i{\'c}, and Pedro Vaz.
\newblock {$\mathfrak{sl}(N)$}-link homology {$(N\geq 4)$} using foams and the
  {K}apustin-{L}i formula.
\newblock {\em Geom. Topol.}, 13(2):1075--1128, 2009.
\newblock \mathscinet{MR2491657} \doi{10.2140/gt.2009.13.1075}
  \arxiv{0708.2228}.

\bibitem[Rou08]{0812.5023}
Rapha\"{e}l Rouquier.
\newblock {$2$-Kac-Moody} algebras, 2008.
\newblock \arxiv{0812.5023}.

\bibitem[Sik05]{MR2171796}
Adam~S. Sikora.
\newblock Skein theory for {${\rm SU}(n)$}-quantum invariants.
\newblock {\em Algebr. Geom. Topol.}, 5:865--897 (electronic), 2005.
\newblock \arxiv{math.QA/0407299} \mathscinet{MR2171796}
  \doi{10.2140/agt.2005.5.865}.

\bibitem[TL02]{MR1896470}
Valerio Toledano~Laredo.
\newblock A {K}ohno-{D}rinfeld theorem for quantum {W}eyl groups.
\newblock {\em Duke Math. J.}, 112(3):421--451, 2002.
\newblock \arxiv{math/0009181} \mathscinet{MR1896470}
  \doi{10.1215/S0012-9074-02-11232-0}.

\end{thebibliography}
% ----------------------------------------------------------------
\vspace{1cm}

% ----------------------------------------------------------------
\end{document}